\documentclass[12pt]{article}
\usepackage{amsmath}
\usepackage{amsfonts}
\usepackage{mitpress}
\usepackage{comment}
\usepackage[toc,page, title]{appendix}
\usepackage{graphicx}
\usepackage{color}
\usepackage{marginfix}

\DeclareMathOperator{\Tr}{Tr}

\newcommand{\R}{{\mathbf R}}
\newcommand{\Dom}{\mathop{\rm Dom}}

\newcommand{\spt}{\mathop{\rm spt}}
\newcommand{\dom}{\mathop{\rm dom}}

\newcommand{\p}{{\partial}}

\newcommand{\tr}{\mathop{\rm tr}}
\newtheorem{theorem}{Theorem}

\newtheorem{corollary}[theorem]{Corollary}

\newtheorem{example}[theorem]{Example}

\newtheorem{lemma}[theorem]{Lemma}

\newtheorem{proposition}[theorem]{Proposition}
\newtheorem{remark}[theorem]{Remark}

\newtheorem{assumption}[theorem]{Assumption}

\newenvironment{proof}[1][Proof]{\noindent\textbf{#1.} }{\ \rule{0.5em}{0.5em}}
\newdimen\dummy
\dummy=\oddsidemargin
\begin{document}

\title{Optimal transportation between unequal 
dimensions\thanks{%
The authors are grateful to Toronto's Fields' Institute for the Mathematical
Sciences for its kind hospitality during part of this work. RJM acknowledges
partial support of this research by Natural Sciences and Engineering
Research Council of Canada Grant 217006-15, 
by a Simons Foundation Fellowship, and by the US
National Science Foundation under Grant No.\ DMS-14401140 while in
residence at the Mathematical Sciences Research Institute in Berkeley CA
during January and February of 2016.  He thanks S.-Y. Alice Chang for a stimulating conversation.
BP is pleased
to acknowledge support from Natural Sciences and Engineering Research
Council of Canada Grants 412779-2012 and 04658-2018, as well as a  University of Alberta start-up
grant.  He is also grateful to the Pacific Institute for the Mathematical Sciences, in Vancouver, BC, Canada, for its generous hospitality during his visit in February and March of 2017.  \qquad 
\copyright\ by the authors, \today}}
\author{Robert J McCann\thanks{%
Department of Mathematics, University of Toronto, Toronto, Ontario, Canada
mccann@math.toronto.edu} \ and Brendan Pass\thanks{%
Department of Mathematical and Statistical Sciences, University of Alberta,
Edmonton, Alberta, Canada pass@ualberta.ca.}}

\maketitle

\begin{abstract}
We establish that solving an optimal transportation problem in which the source
and target densities are defined on manifolds with different dimensions, 
is equivalent to solving a  new nonlocal analog of the Monge-Amp\`ere 
equation,  introduced here for the first time.  Under suitable topological conditions,
we also establish that solutions are smooth if and only if a local variant of the
same equation admits a smooth and uniformly elliptic solution. We show that this local equation is 
elliptic, and $C^{2,\alpha}$ solutions can therefore be bootstrapped to obtain higher regularity results, assuming smoothness of the corresponding differential operator, which we prove under simplifying assumptions.
 For one-dimensional targets, our sufficient criteria for regularity of solutions to the resulting ODE
are considerably less restrictive than those required by earlier works.
\end{abstract}

\section{Introduction}
Since the 1980s \cite{CullenPurser84} \cite{KnottSmith84} \cite{RuschendorfRachev90}  
and the celebrated work of Brenier \cite{Brenier87} \cite{Brenier91},  it has been well-understood
\cite{McCann97} that for the quadratic cost $c(x,y)=\frac12 |x-y|^2$ on $\R^n$, 
solving the Monge-Kantorovich optimal transportation 
problem is equivalent to solving a degenerate elliptic Monge-Amp\`ere equation:
that is, given two probability densities $f$ and $g$ on $\R^n$,  the unique optimal map
between them, $F=Du$, is given by a convex solution $u$ to the boundary value problem
\begin{eqnarray}\label{Monge-Ampere}
g \circ Du \det D^2 u 
= f &&\ {\rm [a.e.],} 
\\ Du \in \spt g &&\ {\rm [a.e.],}
\label{2BVP}
\end{eqnarray}
%
where $\spt g \subset \R^n$ is the smallest closed set of full mass for $g$.  
Similarly, its inverse is given by the gradient of the convex solution $v$ to the boundary value problem
\begin{eqnarray}\label{Monge-Ampere2}
f\circ Dv \det D^2 v 
= g &&\ {\rm [a.e.],} 
\\ Dv \in \spt f &&\ {\rm [a.e.]}.
\label{2BVP2}
\end{eqnarray}
Notice the quadratic cost implicitly requires $x$ and $y$ to live in 
the same space.  Subsequent work of Ma, Trudinger and Wang \cite{MaTrudingerWang05} leads to
an analogous result for other cost functions $c(x,y)=-s(x,y)$ satisfying suitable conditions,
still requiring $x$ and $y$ to live in spaces of the same dimension $n$;  see also 
 earlier works such as
\cite{Caffarelli96} \cite{GangboMcCann96}   \cite{Levin99} \cite{McCann01} \cite{Wang04}.
The purpose
of the present article is to explore what can be said when $x \in \R^m$ and $y\in \R^n$ live in spaces 
with different dimensions $m>n$, as in e.g.~\cite{GangboMcCann00} \cite{ChiapporiMcCannPass17}  \cite{Lott17}.

Although the symmetry between $x$ and $y$ is destroyed,
the duality theorem from linear programming, 
 \cite{Kuhn55} \cite{Santambrogio15} \cite{Brezis18},
strongly suggests that the problem
can still be reduced to finding a single scalar potential $u(x)$ or $v(y)$ reflecting the relative
scarcity of supply $f$ at $x$ (or demand $g$ at $y$).  Although this potential solves a
minimization problem, it is not clear what equation, if any,
selects it.  Nor whether one expects its solution to be smoother
than Lipschitz and semiconvex \cite{GangboMcCann96}  \cite{FigalliGigli11}.  
These are among the questions addressed hereafter.  Our primary results
are as follows:  We exhibit an integro-differential equation which selects $v(y)$. In contradistinction
to the case investigated by 
Ma, Trudinger and Wang,  our equation, though still fully nonlinear,  is in general nonlocal.  
 However, we also show this equation has two local analogs, one of which is at least degenerate-elliptic.
These may or may not admit solutions:  however under mild topological conditions,  it turns out they admit a
$C^2$ smooth, strongly elliptic solution if and only if the dual  linear program admits $C^2$ minimizers. 
These locality criteria  build upon our results with  Chiappori \cite{ChiapporiMcCannPass17} 
from $n=1$,  and extend the notion of nestedness introduced there to targets of arbitrary dimension.
We also relax and refine the notion of nestedness, leading to regularity results for a large new class of examples even when $n=1$.

Our basic set-up is as follows.  Fix $m\ge n\ge 1$ and sets 
 $X \subset \R^m$ and  $Y \subset \R^n$ equipped with Borel probability
densities $f$ and $g$.  We say $F:X \longrightarrow Y$ pushes $f$ forward to 
$g=F_\#f$ if $F$ is Borel and
\begin{equation} \label{eqn: push-forward for densities}
\int_{Y} \psi(y) g(y) dy  = \int_X \psi(F(x)) f(x) dx,
\end{equation}
for all bounded Borel test functions $\psi \in L^\infty({Y})$.  
If, in addition,  $F$ happens to be Lipschitz
and its (n-dimensional) Jacobian $JF(x):=\det^{1/2} [DF(x) DF^T(x)]$ vanishes at most
on a set of $f$ measure zero,  then the co-area formula yields
\begin{equation}\label{co-area}
g(y) = \int_{F^{-1}(y)} \frac{f(x)}{JF(x)} d{\mathcal H}^{m-n}(x)
\end{equation}
for a.e. $y \in Y$, where ${\mathcal H}^k$ denotes $k$-dimensional Hausdorff measure.

Given a surplus function $s \in C^2(X \times \bar Y)$,  Monge's problem is to compute
\begin{equation}\label{Monge}
\bar s(f,g) := \sup_{F_\#f=g} \int_X s(x,F(x)) f(x) dx,
\end{equation}
where the supremum is taken over maps $F$ pushing $f$ forward to $g$.
The supremum is well-known to be uniquely attained provided $X \times Y$ is open and $s$ is {\em twisted}
\cite{Villani09},
meaning  $D_x s(x,\cdot)$ acts injectively on $\bar Y$ for each $x \in X$;
here $\bar Y$ denotes the closure of $Y$.
It can be characterized through the Kantorovich dual problem 
\begin{equation}\label{Kantorovich dual}
\bar s(f,g) = \min_{u(x) + v(y) \ge s(x,y)} \int_{X} u(x) f(x) dx + \int_Y v(y) g(y) dy, 
\end{equation}
where the minimum is taken over pairs $(u,v) \in L^1(f) \oplus L^1(g)$ satisfying
$u \oplus v \ge s$
 throughout $X \times  Y$.  Dual minimizers of the form $(u,v) = (v^s,u^{\tilde s})$
are known to exist \cite{Villani09}, where
\begin{equation}\label{s-dual}
v^s(x) = \sup_{y \in \bar Y} s(x,y) - v(y) \qquad u^{\tilde s}(y) = \sup_{x \in X} s(x,y) - u(x).
\end{equation}
Such pairs of payoff functions are called $s$-conjugate,  and $u$ and $v$ are said to be $s$- and
$\tilde s$-convex, respectively. 

To motivate our first result,   let $X \subset \R^m$  be open and $Y \subset \R^n$ be open and bounded,
 and $s\in C^2(X \times \bar Y)$ twisted and {\em non-degenerate},
meaning  in addition to the injectivity of $y \in \bar Y \mapsto D_xs(x,y)$ mentioned above that
$D^2_{xy} s(x,y)$ has maximal rank throughout $X \times \bar Y$.
Suppose $F$ maximizes the primal problem \eqref{Monge} and 
$(u,v)=(v^s,u^{\tilde s})$ are $s$-convex payoffs minimizing the 
dual problem \eqref{Kantorovich dual}.  Then $u(x) + v(y) - s(x,y) \ge 0$
on $X \times \bar Y$,  with equality on graph$(F)$.  
Thus
\begin{eqnarray}\label{s-subdifferential}
F^{-1}(y) 
&\subset& \p_{\tilde s} v(y) 
\\ &:=& \{ x \in X \mid  
s(x,y) - v(y) =\sup_{y' \in \bar Y} s(x,y') - v(y') 
\}.
\end{eqnarray}
Since $s \in C^2(X \times \bar Y)$, $u$ and $v$ admit second-order Taylor
expansions Lebesgue a.e. as in e.g. \cite{GangboMcCann96} \cite{Villani09}, and 
the first- and second-order conditions for equality on graph$(F)$ imply
\begin{eqnarray}\label{FOCv}
Dv(F(x))  &=& D_y s(x,F(x)) \qquad \mbox{\rm [$f$-a.e.] and}
\\ D^2 v(F(x)) &\ge& D^2_{yy} s(x,F(x)) \qquad \mbox{\rm [$f$-a.e.]}.
\label{SOCv}
\end{eqnarray}
Differentiating the first-order condition yields
\begin{equation}\label{JACv}
[D^2 v(F(x))-D^2_{yy} s(x,F(x))] DF(x)  = D^2_{xy} s(x,F(x))  \qquad \mbox{\rm [$f$-a.e.]}
\end{equation}
as in e.g. \cite{MaTrudingerWang05}.
Since $D^2_{xy} s$ has full-rank, when $F$ happens to be Lipschitz
we identify 
its Jacobian $f$-a.e.\ as
\begin{equation}\label{Jacobian}
JF(x) = \frac{\sqrt{\det [D^2_{xy} s(x,F(x)) (D^2_{xy} s(x,F(x)))^T]}}{\det [D^2 v(F(x))-D^2_{yy} s(x,F(x))]}.
\end{equation}
In this case we can rewrite \eqref{co-area} in the form
\begin{equation}\label{near PDE}
g(y) = \int_{F^{-1}(y)} \frac{\det [D^2 v(y)-D^2_{yy} s(x,y)]}{\sqrt{\det D^2_{xy} s(x,y) (D^2_{xy} s(x,y))^T} } 
f(x) d{\mathcal H}^{m-n}(x).
\end{equation}

Except for the appearance of the map $F$ in the domain of integration,  this would be a partial differential equation relating $v$ 
to the data $(s,f,g)$.  However, using  twistedness of the  surplus 
we'll show that for a.e.~$y$,
the containment \eqref{s-subdifferential} is saturated up to an ${\mathcal H}^{m-n}$ negligible set.
%
%
Thus we arrive at 
\begin{equation}\label{nonlocal PDE}
g(y) = \int_{\p_{\tilde s}v(y)} \frac{\det [D^2 v(y)-D^2_{yy} s(x,y)]}{\sqrt{\det D^2_{xy} s(x,y) (D^2_{xy} s(x,y))^T} } 
f(x) d{\mathcal H}^{m-n}(x) \quad \mbox{[\rm ${\mathcal H}^n$-a.e.].}
\end{equation}

This is an analog of the Monge-Amp\`ere equation \eqref{Monge-Ampere},  
familiar from the case $s(x,y) = -\frac12 |x-y|^2$, or equivalently $s(x,y) = x \cdot y$.  
Notice the boundary condition \eqref{2BVP}
for that case is automatically subsumed in formulation \eqref{nonlocal PDE}.
However, unlike the case $m=n$, it is badly nonlocal since the domain of integration $\p_{\tilde s} v(y)$ 
defined in \eqref{s-subdifferential}
may potentially depend on $v(y')$ for all $y' \in Y$.

For twisted non-degenerate $s$ and an $s$-convex $v$, our first result states that $v$ 
satisfies \eqref{nonlocal PDE} if and only if $v$ combines with its conjugate $u=v^s$ 
to minimize \eqref{Kantorovich dual}; see Corollary \ref{C:nonlocal PDE} of \S\ref{S:nonlocalPDE}.
Since the optimal map $F$ can be recovered from the first-order condition
\begin{equation}\label{FOCu}
D_x s(x,F(x))= Du(x),
\end{equation}
analogous to \eqref{FOCv},
this shows Monge's problem has been reduced to the solution of the partial differential equation
\eqref{nonlocal PDE} for the $\tilde s$-convex scalar function $v$.
Note that although we neither assume nor establish Lipschitz continuity of $F$ in the
sequel, for $s \in C^2$ twisted the $s$-convexity of $u$ makes $F$ countably Lipschitz,
as in e.g.~\cite{Santambrogio15}.

Although non-locality makes this equation a challenge to solve,
it turns out there is a class of problems for which \eqref{nonlocal PDE} can be replaced
by a local partial differential equation, as follows.  Introduce the $m-n$ dimensional submanifold
\begin{eqnarray*}
X_1(y,p,Q) &:=& X_1(y,p)  := \{x\in X \mid D_y s(x,y) = p\}
\end{eqnarray*}
 of $X$ and its closed subset
\begin{eqnarray}
 X_2(y,p,Q)  &:=& \{x\in X_1(y,p) \mid D^2_{yy} s(x,y) \le Q\}.
\label{potential indifference sets}
\end{eqnarray}
Now \eqref{FOCv}--\eqref{SOCv} imply
\begin{equation}\label{inclusion}
\p_{\tilde s} v(y) \subset X_2(y,Dv(y),D^2 v(y)) \subset X_1(y,Dv(y))
\end{equation}
for all $y \in \dom D^2 v$,  the subset of $\bar Y$ where $v$ admits a second-order Taylor expansion. 
It is often the case that one or both of these containments becomes an equality,
at least up to ${\mathcal H}^{m-n}$ negligible sets.
In this case locality is restored: 
we can then write \eqref{nonlocal PDE} in the form
\begin{equation}\label{local PDE}
 G(y, Dv(y),D^2v(y)) = g(y) \quad {\rm [a.e.]}, 
\end{equation}
where
\begin{eqnarray}\label{G_i}
G(y,p,Q) &:=& G_i(y,p,Q)
\\ &:=&\int_{X_i(y,p,Q)} \frac{\det [Q-D^2_{yy} s(x,y)]}{\sqrt{\det D^2_{xy} s(x,y) (D^2_{xy} s(x,y))^T} } 
f(x) d{\mathcal H}^{m-n}(x)
\nonumber
\end{eqnarray}
and either $i=1$ or $i =2$.

Our second result states any  classical $s$-convex solution $v \in C^2(Y)$ to either local problem \eqref{local PDE} 
also solves the nonlocal one \eqref{nonlocal PDE};  Corollary~\ref{C:local PDE}.  
Assuming connectedness of $X_1(y,Dv(y))$, 
we show such a solution exists and satisfies the uniform
ellipticity criterion $D^2 v - D^2_{yy} s >0$ if and only if the dual minimization \eqref{Kantorovich dual} admits
a $C^2$ solution ; Theorem \ref{thm: smoothness implies nestedness} of \S\ref{S:localPDE}. 
For an $n=1$ dimensional target,
necessary and sufficient conditions for the more restrictive variant $i=1$ to
admit an $\tilde s$-convex solution have been given
in joint work with Chiappori \cite{ChiapporiMcCannPass17}.
There the ordinary differential equation \eqref{local PDE} is also analyzed 
to show $v$ inherits smoothness from suitable conditions on the data $(s,f,g)$
 in this so-called {\em nested} case.
The existence of a solution to \eqref{local PDE} with $i=1$ extends the notion of nestedness
from $n=1$ to higher dimensions.

We go on to show that the operator $G_2$ is degenerate elliptic in  \S \ref{S:ellipticity}, and 
that the ellipticity is strict at points where $G_2 >0$.  
As a consequence,  
we are able to deduce higher regularity of solutions $v$ of \eqref{local PDE} with $i=2$ from $C^{2,\alpha}$ regularity
 in Theorem \ref{thm: bootstrapping}, provided $G_2$ is sufficiently smooth.
 In Theorem \ref{thm: smoothness of G} of \S \ref{S:G1=G2}
we establish this smoothness for the simpler operator $G_1$, allowing for the passage from $C^{2,\alpha}$ to higher regularity when $G_2=G_1$.   For one-dimensional targets, we establish the smoothness of $G_2$ in
Theorem \ref{thm: smoothness of G_2} of \S \ref{S:G2}, whether or not it coincides with $G_1$.
The hypothesized second order smoothness and uniform  ellipticity of $v$ remain intriguing open questions --- with partial resolutions known only in the cases $n=m$ of Ma, Trudinger and Wang \cite{MaTrudingerWang05} \cite{TrudingerWang09b} 
(which built on earlier work of Caffarelli \cite{Caffarelli92} \cite{Caffarelli96b}, Delanoe \cite{Delanoe91} and Urbas \cite{Urbas97}),  and for $n=1$ in the nested case~\cite{ChiapporiMcCannPass17};  to these we now add the non-nested cases which satisfy the local equations
\eqref{local PDE}--\eqref{G_i} with $(n,i)=(1,2)$, resolved in \S \ref{S:ODE} below.  When regularity fails for $m=n$ the size of the 
singular set has been estimated by DePhilippis and Figalli \cite{DePhilippisFigalli15}, building on work of 
Figalli \cite{Figalli10} with Kim \cite{FigalliKim10}; for related results see Kitagawa and Kim \cite{KimKitagawa16}
and the survey \cite{DePhilippisFigalli14}.




\section{A nonlocal partial differential equation for optimal transport}
\label{S:nonlocalPDE}

Given $X \subset \R^m$ and $Y \subset \R^n$, a Borel probability density
$f$ on $X$ and a Borel map $F:X \longrightarrow Y$,  we 
define the pushed-forward measure $\nu := F_\#f$ by 
\begin{equation}\label{push-forward}
\int_{Y} \psi(y) d\nu(y)  = \int_X \psi(F(x)) f(x) dx
\end{equation}
for all bounded Borel functions $\psi \in L^\infty(Y)$.  This definition
extends \eqref{eqn: push-forward for densities} to the case where $\nu$ need not be absolutely continuous
with respect to Lebesgue; however when $\nu$ is absolutely continous with Lebesgue density
$g$, we abuse notation by writing $g=F_\#f$.



Recall $s\in C^2(X \times \bar Y)$ is {\em twisted} if for each $x \in X$ the map
$y \in \bar Y \mapsto D_x s(x,y)$ is one-to-one.  If
$$
D_x s(x,y) = p
$$
we can then deduce $y$ uniquely from $x$ and $p$, in which case we write 
$y= s$-$\exp_x p := D_x s(x,\cdot)^{-1}(p)$.
The non-degeneracy of $s$ (full-rank of $D^2_{xy}s$) 
guarantees $s$-$\exp$ is a continuously differentiable function of $(x,p)$ where 
defined, by the implicit function theorem.  Thus for a twisted cost function,  the first-order condition
\eqref{FOCu} allows us to identify the map $F=s$-$\exp \circ Du$
at points of $X$ where $u$ happens to be differentiable.
We denote the set of such points by $\dom Du$.  Similarly we denote the set of 
points where $F:X \longrightarrow \bar Y$ is approximately differentiable by 
$\dom DF$,  and the set where $u$ admits a second order Taylor expansion by $\dom D^2 u$.  
When $s$ is non-degenerate and twisted, \eqref{FOCu} implies $\dom DF = \dom D^2 u$.  
A function $u:X \subset \R^m$ is said to be {\em semiconvex} if there exists $k \in \R$ such that
$u(x) + k|x|^2$ is the restriction to $X$ of a convex function on~$\R^m$.
 
\begin{theorem}[Properties of potential maps]\label{T:nonlocal PDE} 
Fix $m\ge n$, open sets $X \subset \R^m$ and $Y \subset \R^n$ with $Y$ bounded, and
$s \in C^2(X \times \bar Y)$ (so $\|s\|_{C^2(X \times \bar Y)}<\infty$)
 twisted and non-degenerate.  Any pair 
$(u,v)=(v^s,u^{\tilde s})$ of $s$-conjugate functions \eqref{s-dual} 
are semiconvex, Lipschitz, and have second-order Taylor expansions Lebesgue a.e.
The map $F:\dom Du \longrightarrow \bar Y$ satisfying \eqref{FOCu} is unique 
and differentiable Lebesgue a.e. 
Decompose $\bar Y$ into $Y_+:= \dom D^2 v \subset \bar Y$
and $Y_- = \bar Y \setminus Y_+$ and set $X_\pm := F^{-1}(Y_\pm)$.
The Jacobian $JF(x):=\det^{1/2}[DF(x)DF(x)^T]$  
is positive on $X_+ \cap \dom DF$ and given there by 
\begin{equation}\label{Jacobian2}
JF(x)%
= \frac {\sqrt{\det [D^2_{xy} s (x,F(x)) D^2_{xy}s(x,F(x))^T]}}{\det[D^2 v(F(x)) - D^2_{yy}s(x, F(x))]}.
\end{equation}

Any Borel probability density on $X$ can be decomposed as $f=f_+ + f_-$ where $f_\pm =f 1_{X_\pm}$
are mutually singular.  Their images $F_\# (f_\pm)$ are measures living on the 
disjoint sets $Y_\pm$.  
Here $F_\# (f_+)$ is absolutely continuous with 
respect to Lebesgue: its density given for Lebesgue a.e.\ $y \in \bar Y$ by
\begin{eqnarray}
\label{change of variables+}
g_+(y) 
&=& \int_{F^{-1}(y)} \frac{f_+(x)}{JF(x)} d{\mathcal H}^{m-n}(x)
\\&=&  \int_{\p_{\tilde s} v(y)} 
 \frac{\det [D^2 v(y)-D^2_{yy} s(x,y)]}{\sqrt{\det D^2_{xy} s(x,y) (D^2_{xy} s(x,y))^T} } f(x) d{\mathcal H}^{m-n}(x).
\label{change of variables2}
\end{eqnarray}
\end{theorem}

\begin{proof}
It is well-known that $u=v^s$ and $v=u^{\tilde s}$ are Lipschitz and semiconvex \cite[Lemma 3.1]{McCannGuillen13}: 
they inherit distributional bounds such as $|Du| \le \sup_Y |D_x s|$ and 
$D^2 u \ge \inf_Y D^2_{xx} s$ 
from $s\in C^{2}$ .
This implies they extend continuously to $\bar X$ and $\bar Y$, where they are twice differentiable 
a.e.\ by Alexandrov's theorem \cite[Theorem 14.25]{Villani09};
indeed, for $x_0 \in \dom D^2 u$ we have
\begin{equation}\label{Alexandrov}
0 = \lim_{x \to x_0} \sup_{p \in \p u(x)} \frac{p - Du(x_0) - D^2 u(x_0)(x-x_0) }{|x-x_0|} 
\end{equation}
which asserts differentiability (rather than just approximate differentiability) of $Du$ at $x_0$.

Recall $u(x) + v(y) - s(x,y) \ge 0$ on $X \times \bar Y$.  For each $x \in \dom Du$ at least one $y \in \bar Y$ produces equality,
since the maximum \eqref{s-dual} defining $v^s(x)$ is attained.   This $y$ satisfies 
the first order condition $D_x s(x,y) = Du(x)$,  which identifies it as $y=F(x)$ by the twist condition.
We abbreviate $F = s$-$\exp \circ Du$.
%
%
We note $Du$ is differentiable a.e.\ in a neighbourhood of $x\in \dom F$,
and the map $s$-$\exp$ is well-defined and continuously differentiable in a neighourhood of $(x,Du(x))$ 
by the twist and non-degeneracy of $s$.
Recall for any $\epsilon>0$ the semiconvex function $u$ 
agrees with a $C^2$ smooth function $u_\epsilon$ outside a set
of volume $\epsilon$.  As a result we see the extension $\bar F$ is $C^1$ in a neighbourhood of $\dom Du$, 
except on a set of arbitrarily small
volume,  hence is countably Lipschitz (and approximately differentiable Lebesgue a.e.)
The fact that it is actually differentiable a.e.\ follows from $s$-$\exp \in C^1$ and \eqref{Alexandrov}.




Since $u(x) + v(y) - s(x,y) \ge 0$ vanishes at $(x,F(x)) \in X \times \bar Y$ 
for each $x\in X_+$,  we  can differentiate \eqref{FOCv} if $x \in \dom DF$ 
to obtain \eqref{JACv}.
Since the right hand side has rank $n$ we conclude both factors on the 
left must have rank $n$ as well.  This shows $J F(x)>0$ and noting \eqref{SOCv} establishes \eqref{Jacobian2}.

Decomposing a probability density $f=f_++f_-$ on $X$ into $f_\pm = f1_{X_\pm}$,
the statements of mutually singularity follow from $Y_+ \cap Y_- = \emptyset = X_+ \cap X_-$.
Now decompose $X \setminus X_\infty = \cup_{i=1}^\infty X_i$ into countably many disjoint Borel sets 
$X_i \subset \R^m$ on which $F$ is $C^1$ with $JF(x)>1/i$ on $X_i$,  
plus an $f_+$ negligible set $X_\infty$.  Let $f_i=f_+  1_{X_i}$ denote
the restriction of $f_+$ to $X_i$,  and $g_i := F_\# f_i$ the density of the push-forward of $f_i$.  
It costs no generality to assume $f=0$ on $X_\infty \cup \p X$.
Let $F_i$ denote
a Lipschitz extension of $F$ from $X_i$ to $\R^m$.  For each $\phi \in L^1(\R^m)$, the co-area formula 
\cite[\S 3.4.3]{EvansGariepy92} implies
$$
\int_{X_i} \phi(x) JF_i(x) dx = \int_{\R^n} dy \int_{X_i\cap F_i^{-1}(y)} \phi d{\mathcal H}^{m-n}.
$$
Given $\psi \in L^\infty(\R^n)$ with bounded support ensures 
$\phi = f_i \psi \circ F_i / JF_i \in L^1(\R^m)$ hence
\begin{eqnarray*}
\int_{\R^n} g_i \psi 
&=& \int_{\R^m} f_i \psi \circ F_i
\\ &=&  \int_{\R^n} dy \psi(y) \int_{X_i \cap F_i^{-1}(y)} \frac{f_i}{JF_i} d{\mathcal H}^{m-n}.
\end{eqnarray*}
Recalling $F_i=F$ on $X_i$, we infer
$$
g_i(y) = \int_{F^{-1}(y)} \frac{f_i(x)}{J F(x)} d{\mathcal H}^{m-n}(x).
$$
a.e.\ since $\psi \in L^\infty$ had bounded support but was otherwise arbitrary.
Summing on $i$, the disjointness of $X_i$ and the fact that $f =0$ on $X_\infty$
yields \eqref{change of variables+}. 

Now,  $y\in Y_+$ implies $f=f_+$ on $F^{-1}(y) \subset X_+$.  Since $Y_+ = \dom D^2 v$ has
full measure in $Y$,  we can replace $f$ by $f_+$ in \eqref{change of variables+}.  On the other hand,
$X_+ \subset \p_{\tilde s} v(Y_+)$ with the difference satisfying
$\p_{\tilde s} v(Y_+)\setminus X_+\subset \bar X \setminus \dom Du$.
This shows $f$ vanishes on $\p_{\tilde s} v(y) \setminus F^{-1}(y)$.  Thus we can also replace $F^{-1}(y)$
by $\p_{\tilde s} v(y)$. 
Finally, \eqref{Jacobian2} relates \eqref{change of variables+} to \eqref{change of variables2}.
\end{proof}

\begin{corollary}[Equivalence of optimal transport to nonlocal PDE]\label{C:nonlocal PDE}
Under the hypotheses of Theorem \ref{T:nonlocal PDE},  let $f$ and $g$ denote probability densities on $X$ 
and $Y$.  If $v=v^{s\tilde s}$ satisfies the nonlocal equation \eqref{nonlocal PDE} then $(v^s,v)$ minimize 
Kantorovich's dual problem
\eqref{Kantorovich dual}.  Conversely,  if $(u,v)=(v^s,u^{\tilde s})$ minimize \eqref{Kantorovich dual}
then $v$ satisfies \eqref{nonlocal PDE}.
\end{corollary}

\begin{proof}
First suppose $v=v^{s\tilde s}$ satisfies the nonlocal PDE \eqref{nonlocal PDE}.
Setting $u=v^s$ implies for each $x \in  \dom Du$ the inequality
\begin{equation}\label{constraint}
u(x) + v(y) - s(x,y) \ge 0
\end{equation}
is saturated by some $y \in \bar Y$.    Identifying $F(x) = y$ we have the first-order condition
\eqref{FOCu},  whence $F=s$-$\exp \circ Du$ on $\dom Du$.
We claim it is enough to show $F_\#f=g$:  if so, 
integrating
$$u(x) + v(F(x)) = s(x,F(x))$$
against $f$ yields
$$
\int_{X} u f +  \int_Y v g = \int_{X} s(x,F(x)) f(x) dx,
$$
which in turn shows $F$ maximizes \eqref{Monge} and $(u,v)$ minimizes \eqref{Kantorovich dual}
as desired.  Comparing \eqref{nonlocal PDE} with \eqref{change of variables2} we see $g_+=g$
is a probability measure, hence has the same total mass as $f$.  This implies $g_-=0$
and $F_\#f= g$ as desired.

Conversely,  suppose $(u,v) = (v^s,u^{\tilde s})$ minimizes \eqref{Kantorovich dual}.
Since twistedness of $s$ implies \eqref{Monge} is attained,  there is some map $F:X \longrightarrow \bar Y$
pushing $f$ forward to $g$ such that \eqref{constraint} becomes an equality $f$-a.e. on $Graph(F)$.
This ensures $F=s$-$\exp \circ Du$ holds  $f$-a.e.  
Since $Y_+ := \dom D^2 v \subset \bar Y$ is a set of full measure for $g$,
we conclude $X_+ = F^{-1}(Y_+)$ has full measure for $f$,  whence $f_+ := f1_{X_+}=f$
and $g_+:= F_\#(f_+)=g$.  Now \eqref{nonlocal PDE} follows from \eqref{change of variables2}
as desired.
\end{proof}

\begin{corollary}[Optimal transport via local PDE]\label{C:local PDE}
Under the hypotheses of Theorem \ref{T:nonlocal PDE},  let $f$ and $g$ denote probability densities on $X$ 
and $Y$.  Fix $i=2$ and let $v=v^{s\tilde s}$ have the property that $(s$-$\exp \circ D v^s)_\# f$ vanishes
on $\bar Y \setminus \dom D^2 v$ (as when e.g.~$v \in C^2(\bar V)$).  If 
the local equation \eqref{local PDE} holds ${\mathcal H}^n$-a.e.
then $(v^s,v)$ minimize Kantorovich's dual problem
\eqref{Kantorovich dual}.   
\end{corollary}

\begin{proof}
Fix  $i=2$ and suppose $v=v^{s\tilde s}$ satisfies the local PDE \eqref{local PDE}.
As in the preceding proof, 
setting $u=v^s$ implies for each $x \in \dom Du$ the inequality
\begin{equation} 
u(x) + v(y) - s(x,y) \ge 0
\end{equation}
is saturated by some $y \in \bar Y$.   
Setting $F(x)=y$ we have the first-order condition
\eqref{FOCu},  whence $F= s$-$\exp_x \circ Du$ on $\dom Du$.  

The present hypotheses assert the $Y_+=\dom D^2 v$ forms a set of full measure for $F_\#f$.
Thus $f_-=0$, while $f=f_+$ and $g_+$ are both probability densities
in Theorem \ref{T:nonlocal PDE}.  
Recalling $\p_{\tilde s} v(y) \subset X_2(y,Dv(y),D^2v(y))$ from \eqref{inclusion}, 
we deduce $g \ge g_+$ by comparing \eqref{local PDE} with \eqref{change of variables2}.
Since both densities integrate to $1$, this implies $g=g_+$ a.e.   Thus \eqref{nonlocal PDE}
is actually satisfied and Corollary \ref{C:nonlocal PDE}
asserts $(v^s, v)$ minimizes \eqref{Kantorovich dual}.
\end{proof}



\section{Local PDE from optimal transport}
\label{S:localPDE}

As a partial converse to the preceding corollary, we assert that for either the more restrictive 
($i=1$) or less restrictive ($i=2$) local partial differential 
equation \eqref{local PDE} to admit solutions, it is sufficient that the Kantorovich dual
problem have a smooth minimizer $(u,v)$, with connected potential indifference sets $X_i(y,Dv(y),D^2 v(y))$
--- in which case $v$ also solves \eqref{local PDE}.

\begin{theorem}[When a smooth minimizer implies nestedness]\label{thm: smoothness implies nestedness}
Fix $m\ge n$, probability densities $f$ and $g$ on open sets $X \subset \R^m$ and $Y \subset \R^n$ 
with $Y$ bounded, and $s \in C^2(X \times \bar Y)$ twisted and non-degenerate.  Let $i \in \{1,2\}$. If
$(u,v)=(v^s,u^{\tilde s}) \in C^2(X) \times C^2(Y)$ minimizes the Kantorovich dual \eqref{Kantorovich dual}
and $X_i(y, Dv(y), D^2 v(y))$ is connected for $g$-a.e.\ $y\in Y$,  then the local
equation \eqref{local PDE} holds g-a.e.
\end{theorem}
\begin{proof}
Corollary \ref{C:nonlocal PDE} implies $v$ solves the non-local equation \eqref{nonlocal PDE}.
The local equation $G=g$ 
follows from equality in the inclusion 
\begin{equation}\label{i-inclusion}
\p_s v(y) \subset X_i(y,Dv(y),D^2v(y))
\end{equation}
from \eqref{inclusion} 
for  $\mathcal H^n$-a.e. $y$.
  
We now derive this equality for all $y'\in Y$ with $\p_s v(y')$ non-empty and $X_i':=X_i(y',Dv(y'),D^2 v(y'))$ connected. 
This set contains the full mass of $g$.

Observe both $\p_s v(y')$ and $X_i'$ are relatively closed subsets of $X$.
Thus $\p_s v(y')$ is also closed relative to $X_i'$.   To show it is relatively open, 
let $x' \in \p_s v(y')$.
Since $u,v\in C^2$ we see $F \in C^1(X)$ and $DF$ has full rank at $x'$.
By the Local Submersion Theorem \cite{GuilleminPollack74},  this means we can find 
a $C^1$ coordinate chart on a neighbourhood $U \subset X$ of $x'$ in which $F$ acts as the canonical
submersion:  $F(x_1,\ldots,x_n,x_{n+1},\ldots,x_m) = (x_1,\ldots, x_n)$.  In these coordinates,
$$[ \{y'\} \times \R^{m-n} ]\cap U = F^{-1}(y') \cap U \subset \p_s v(y') \cap U \subset X_i' \cap U \subset X_1' \cap U
$$ 
follows from \eqref{i-inclusion}.
But Proposition 2 of \cite{ChiapporiMcCannPass17} shows
$X_1'$ to be an $m-n$ dimensional submanifold of $X$,  so equality must hold in this
chain of inclusions (at least if $U$ is a ball in the new coordinates).  
This shows $x'$ lies in the interior of $\p_s v(y')$ relative to $X_i'$,
concluding the proof that $\p_s v(y')$ is relatively open.  Thus $\p_s v(y') = X_i'$ since the former
is open, closed and non-empty and the latter is connected.  Equality in \eqref{i-inclusion}
has been established for $g$-a.e.\ $y$, concluding the proof.
\end{proof}



The following example shows that the level set connectivity assumption is required to deduce nestedness; it also illustrates why it may be necessary to consider the  $i=2$ case of the local equation.
\begin{example}[Annulus to circle]
Consider transporting uniform mass on the  annulus, $X=\{x \in \mathbf{R}^2: 1/2 \leq |x| \leq 1\}$ to uniform measure on the circle, $Y =\{y \in \mathbf{R}^2:  |y| = 1\}$ with the bilinear surplus, $s(x,y) =x \cdot y$.  It is easy to see that $x \cdot y \leq |x|$, with equality only when $y=\frac{x}{|x|}$, implying that the optimal map takes the form $F(x) =\frac{x}{|x|}$ and the potentials $u(x) =|x|$, $v(y) =0$.   These are smooth on the regions $X$ and $Y$.  Note that $X_2(y, Dv(y), D^2 v(y))=\{x \in X: \frac{x}{|x|}=y\}$ is connected and coincides with $\partial^sv(y)$ (as is guaranteed by the preceding theorem).  On the other hand,    $X_1(y, Dv(y), D^2 v(y))=\{x \in X: \frac{x}{|x|}=y\}\cup \{x \in X: \frac{x}{|x|}=-y\}$ is disconnected and the inclusion $ \partial^sv(y) \subseteq X_1(y, Dv(y), D^2 v(y)) $ is strict.
\end{example}

\section{Concerning the regularity of maps}
\label{S:maps}

This section collects some conditional results which illustrate how strong $s$-convexity of $v$
plus a connectedness condition can imply the continuity and differentiability of optimal maps.
In the case of  equal dimensions, a related connectedness requirement appears in work of 
Loeper \cite{Loeper09}.
This section is purely $s$-convex analytic; no measures are mentioned.

\begin{lemma}[Continuity of maps (local)]\label{L:continuity}
	Fix $m\ge n$, open sets $X \subset \R^m$ and $Y \subset \R^n$ 
	with $Y$ bounded, and $s \in C^2(X \times \bar Y)$ twisted and non-degenerate.
	Let  $(u,v)=(v^s,u^{\tilde s})$ and $D^2 v(y) > D^2_{yy} s(x,y)$ for some 
	$(x,y) \in X \times [\p_s v^s(x) \cap \dom D^2 v]$.  The isolated point $y$ forms a $C^1$-path-connected
	component of $\p_s u(x)$. 
	Thus $x \in \dom Du$ if, in addition, $\p_s u(x)$ is $C^1$-path-connected.
\end{lemma}

\begin{proof}
	Fix $(u,v)$ and $(x,y)$ as in the lemma.
	Let $y: t\in[0,1] \mapsto y(t) \in \p_s u(x)$ be a continuously differentiable curve departing from $y(0)=y$ 
	with non-zero velocity $y'(0) \ne 0$.  Since the non-negative function $u(x) + v(\cdot) - s(x,\cdot) \ge 0$
	vanishes on this curve,  differentiation shows $y'(0)$ to be in the nullspace of 
	$D^2 v(y) - D^2_{yy} s (x,y)$.  This contradicts the positive-definiteness assertion and shows
	no such curve can exist.
	
	Thus $C^1$-path connectedness implies $\p_s u(x) = \{ y\}$.  The semiconvexity
	of $u$ shown in Theorem \ref{T:nonlocal PDE} implies $x \in \dom Du$ provided
	we can establish convergence of $Du(x_k)$ to a unique limit  
	whenever $x_k \in \dom Du$ converges to $x$.  Therefore, let $x_k \in \dom Du$ converge to $x$,
	and choose $y_k \in \p_s v(x_k)$. Any accumulation point $y_\infty$ of the $y_k$ satisfies
	$y_\infty \in \p_s v(x)=\{y\}$.  Now letting $k\to \infty$ in $Du(x_k) = D_x s(x_k,y_k)$ 
	yields $Du(x_k) \to D_x s(x,y)$ to establish $x \in \dom Du$.
\end{proof} 

\begin{corollary}[Continuity of maps (global)]\label{C:continuity}
	Fix $m\ge n$, open sets $X \subset \R^m$ and $Y \subset \R^n$ with $Y$ bounded, and 
	$s \in C^2(X \times \bar Y)$ twisted and non-degenerate.
	Let  $(u,v)=(v^s,u^{\tilde s})$ with $v \in C^2(\bar Y)$.  Then $u \in C^1(X)$ if 
	for each $x\in X$: $\p_s u(x)$ is $C^1$-path connected and $D^2 v(y) > D^2_{yy} s(x,y)$ 
	for some $y \in \p_s u(x)$.
\end{corollary}

\begin{proof}
	Lemma \ref{L:continuity} implies $X = \dom Du$ under the hypotheses of Corollary \ref{C:continuity}.
	Since semiconvexity of $u$ was shown in Theorem \ref{T:nonlocal PDE},  this is sufficient to conclude 
	$u \in C^1(X)$.
\end{proof}

\begin{proposition}
	[Criteria for differentiability of maps]
	\label{P:BV}
	Fix $m\ge n$, open sets $X \subset \R^m$ and $Y \subset \R^n$ with $Y$ bounded, and 
	$s \in C^2(X \times \bar Y)$ twisted and non-degenerate.
	Use  $(u,v)=(v^s,u^{\tilde s})$ with $u \in C^1(X)$ to define $F:X \longrightarrow \bar Y$ through
	\eqref{FOCu}.  Then  both 
	$F$ and $D_x s_y (\cdot,F(\cdot))$ are in $(BV_{loc}\cap C) (X,\R^n)$.
	If, in addition, $v\in C^{1,1}(Y)$ then $F \in \Dom D^2 v$ on a set of $|DF|$ full measure,
	and as measures
	\begin{equation}\label{distributional map gradient} 
	(D^2 v(F(x))-s_{yy}(x,F(x)))DF(x) = D_x s_y(x,F(x)).
	\end{equation}
	In this case,  $F$ is Lipschitz in any open subset of $X$ where
	\begin{equation}\label{uniform ellipticity}
	D^2 v(F(x))-D^2_{yy} s(x,F(x)) \ge \epsilon I>0
	\end{equation}
	is uniformly positive definite (and $F$ inherits higher differentiability from $v$ and $s$ in this case).
\end{proposition}

\begin{proof}
	Recalling
	\begin{equation}\label{ufoc}
	Du(x) = D_x s(x,F(x)),
	\end{equation}
	the continuity $F= s$-$\exp \circ Du$ follows from $u\in C^1(X)$ and the twistedness and non-degeneracy of $s$.
	
	Since $u$ from Theorem \ref{T:nonlocal PDE} is also semiconvex, 
	its directional derivatives lie in $BV(X)$ and its gradient in $BV(X,\R^2)$.  
	We shall use \eqref{ufoc}
	to deduce $F \in BV_{loc}(X)$,  which means its directional
	weak derivatives are signed Radon measures on $X$.  Fix $x' \in X$
	and set $y'=F(x')\in Y$.  Since $D^2_{xy} s$ has full rank, we can invert \eqref{ufoc}
	to express
	$$
	F(x) = \left[ 
	D_x s(x,\cdot)\right]^{-1} Du(x)
	$$
	as the composition of a $C^1_{loc}$ map and a componentwise $BV$ map.
	This shows $F \in BV_{loc}(X,\R^n)$ \cite{AmbrosioDalMaso90}.
	
	
	On the other hand, when $Dv$ is assumed Lipschitz, Ambrosio and Dal Maso
	\cite{AmbrosioDalMaso90} assert $F \in \Dom D^2 v$ on a set of $|DF|$ full measure,
	and differentiating 
	$
	Dv(F(x)) = D_ys(x,F(x))
	$
	yields \eqref{distributional map gradient}
	in the sense of measures; 
	$DF$ has no jump part since $F$ is continuous.  
	The fact that $F$ inherits the
	Lipschitz smoothness (and higher differentiability) from $Dv$ follows immediately by rewriting
	\eqref{distributional map gradient}--\eqref{uniform ellipticity} in the form
	$$
	DF(x) = (D^2 v(F(x))-s_{yy}(x,F(x)))^{-1} D_x s_y(x,F(x)) \in L^\infty(X).
	$$\hfill
\end{proof}

\section{Ellipticity and  potential regularity beyond $C^{2,\alpha}$} 
\label{S:ellipticity}

The previous sections show optimal transportation is often
equivalent to solving a nonlinear partial differential equation --- local or nonlocal.
As an application of this reformulation
we show how higher regularity of the solution $v$ on the lower dimensional domain 
can be bootstrapped from its first $2+\alpha$ derivatives. 
This application, though well-known when $n=m$,
 is novel in unequal dimensions. 
It also highlights the need for a theory which explains when $v$ can be expected
to be $C^{2,\alpha}_{loc}$, 
to parallel known results beginning
with \cite{Caffarelli92} \cite{MaTrudingerWang05} for $n=m$;  we identify conditions ensuring this when $n=1$ in the last two sections (see Remark \ref{rem: higher order smoothness}).
Recall that a second-order differential operator $G(y,p,Q)$ is said to be 
{\em degenerate elliptic} if
$G(y,p,Q') \ge G(y,p,Q)$ whenever $Q' \ge Q$, i.e.\ whenever $Q'-Q$ is non-negative definite and both
$Q$ and $Q'$ are symmetric.   
We say the ellipticity is {\em strict} at $(y,p,Q)$  if 
there is a constant $\lambda=\lambda(y,p,Q)>0$ called the {\em ellipticity constant} such that $Q' \ge Q$ implies
\begin{equation}\label{locally uniform ellipticity}
G(y,p,Q') - G(y,p,Q) \ge \lambda \tr [Q'-Q].
\end{equation}

\begin{lemma}[Strict ellipticity]\label{L:strict ellipticity}
The operator $G$ defined by \eqref{potential indifference sets} and \eqref{G_i} with $i=2$ is degenerate elliptic. 
Moreover, if $G(y,p,Q)>0$, and there exists $\Theta > 0$ such that
$Q -D^2_{yy} s(x,y) \le \Theta I$ for all $x \in X_2(y,p,Q)$, then 
the ellipticity constant of $G$ at $(y,p,Q)$ is given by $\lambda = G(y,p,Q)/\Theta$.
\end{lemma}

\begin{proof}
Fixing $(y,p)\in \bar Y \times \R^m$ and $m \times m$ symmetric matrices $Q' \ge Q$,
degenerate ellipticity of $G$ follows from the facts that  $f \ge 0$,
$X_2(y,p,Q) \subset X_2(y,p,Q')$,  and 
$Q' - D^2_{yy} s(x,y) \ge Q - D^2_{yy} s(x,y) \ge 0$ for all $x \in X_2(y,p,Q)$. 

Now suppose also
$Q-D^2_{yy} s(x,y) \le \Theta I<\infty$ for all ${x \in X_2(y,p,Q)}$,
so that  
$$
\tr  [(Q-D^2_{yy} s(x,y))^{-1}(Q'-Q)] \ge \Theta^{-1} {\tr [Q'-Q]}  
$$
for all $Q' \ge Q$.
From here we deduce 
$$
\det [I + (Q-D^2_{yy} s)^{-1}(Q'-Q)] \ge 1 + \Theta^{-1} {\tr [Q'-Q]}.
$$
This can be integrated against $\det[Q-D^2_{yy} s] f  d{\mathcal H}^{m-n}/ \det [D^2_{xy} s D^2_{xy} s^T]$
over $X_2(y,p,Q)$ to find
$$
\frac{G(y,p,Q')}{G(y,p,Q)} \ge 1 + \Theta^{-1} {\tr [Q'-Q]}.
$$
as desired.
\end{proof}

\begin{theorem}[Bootstrapping regularity using Schauder theory]\label{thm: bootstrapping}
Fix  $0<\alpha<1$ and integer $k\ge 2$.
If $g > \epsilon>0$
 on some smooth domain $Y'$ compactly contained in $Y^0$ where  $v \in C^{k,\alpha}(Y')$, and
$G-g \in C^{k-1,\alpha}$ in a neighbourhood $N$ of the 2-jet of $v$ over $Y'$,  then 
\eqref{local PDE}--\eqref{G_i}
with $i=2$ 
implies $ v \in {C^{k+1,\alpha}(Y')}$.
\end{theorem}

\begin{proof}
Since $v \in C^{2,\alpha}(Y')$, \eqref{local PDE} holds in the classical sense.  If $k \ge 3$,
we can differentiate the equation in (say) the $\hat e_k$ direction
to obtain a linear second-order elliptic equation
\begin{equation}\label{linear elliptic inhomogeneous}
a^{ij}(y) D^2_{ij} w + b^i(y) D_{i} w = d(y)
\end{equation}
for $w = \p v / \p y^k$ whose coefficients 
\begin{eqnarray*}
a^{ij}(y) &:=& \frac{\p G }{\p Q_{ij}}\bigg|_{(y,Dv(y),D^2v(y))}  
\\ b^i(y)  &:=& \frac{\p G }{\p q_{i}}\bigg|_{(y,Dv(y),D^2v(y))}
\end{eqnarray*}
and inhomogeneity
$$
d(y) = \frac{\p g }{\p y}\bigg|_{y} 
- \frac{\p G }{\p y}\bigg|_{(y,Dv(y),D^2v(y))} 
$$
have (i) $C^{k-2,\alpha^2}$ norm controlled by $\|G-g\|_{C^{k-1,\alpha}} \|v\|_{C^{k,\alpha}}^\alpha$
and (ii) $C^{k-2,\alpha}$ norm controlled by $\|G-g\|_{C^{k-1,\alpha}} \|v\|_{C^{k,1}}$.  
In case $k=2$, 
we shall argue below that $w \in C^{1,\alpha}$ solves \eqref{linear elliptic inhomogeneous} in the 
{\em viscosity} sense described e.g.\ in \cite{CrandallIshiiLions92}.
From Lemma \ref{L:strict ellipticity} we see the matrix $(a^{ij})$ is bounded below by $\epsilon I / \|v\|_{C^{2}(Y')}$;
it is bounded above by $\|G\|_{C^1(N)}$.  Thus the equation satisfied by $w$ on $Y'$ is uniformly elliptic.
  Since  the coefficient of $w$ vanishes in  \eqref{linear elliptic inhomogeneous},
the Dirichlet problem with continuous boundary data on any ball in $Y'$ is known to admit a unique
(viscosity) solution \cite{CrandallIshiiLions92};  
moreover, this solution is (i) $C_{loc}^{k,\alpha^2}$ 
(by e.g.\ Gilbarg \& Trudinger Theorems 6.13 ($k = 2$) or 6.17 ($k>2$).  
Thus we infer $v \in C^{k+1,\alpha^2}_{loc}(Y')$.  Applying the same argument again starting from the 
improved estimates (ii) now established yields $v \in C^{k+1,\alpha}_{loc}(Y')$.
At this point we have gained the desired derivative of smoothness for $v$;
starting from a neighbourhood slightly larger than $Y'$ yields $v \in C^{k+1,\alpha}(Y')$.
%

In case $k=2$, 
applying the finite difference operator $\Delta^h_k v(y) := [v(y + h\hat e_k) - v(y)]/h$ to the equation
\eqref{local PDE}, the mean value theorem yields $h^*(y) \in [0,h]$ lower semicontinuous such that
\begin{eqnarray*}
0 
&=& \Delta_k^h [G(y,Dv(y),D^2 v(y))-g(y)]
\\ &=&a^{ij}_h(y) D^2_{ij} w_h + b_h^i(y) D_{i} w_h - d_h(y).
\end{eqnarray*}
Here $w_h = \Delta^h_k v$ and the coefficients
\begin{eqnarray*}
a^{ij}_h(y) &:=& \frac{\p G }{\p Q_{ij}}\bigg|_{(I +h^*(y) \Delta_k^h) (y,Dv(y),D^2v(y))}  
\\ b^i_h(y)  &:=& \frac{\p G }{\p q_{i}}\bigg|_{(I + h^*(y)\Delta_k^h)(y,Dv(y),D^2v(y))}
\\ d_h(y) &=& \frac{\p g }{\p y}\bigg|_{y+ h^*(y) \hat e_k} 
- \frac{\p G }{\p y}\bigg|_{(I + h^*(y)\Delta^h_k)(y,Dv(y),D^2v(y))}.
\end{eqnarray*}
are measurable and converge uniformly to $(a^{ij},b^i,d)$ as $h\to 0$.
The solutions $w_h=\Delta_k^h v \in C^{2,\alpha}$, being finite differences, 
converge to $\p v/\p y_k$ in $C^{1,\alpha}(Y')$.
Lemma 6.1 and Remark 6.3 of \cite{CrandallIshiiLions92} show this partial 
derivative $w=\p v/\p y_k$ must then be the required 
viscosity solution of the limiting equation \eqref{linear elliptic inhomogeneous}. 
%
\end{proof}



 Notice $G_2$ is degenerate elliptic even when evaluated on functions which are not $s$-convex. 

\section{On smoothness of  the nonlinear operators $G_i$}\label{sect: smoothness of $G_i$}
\label{S:G1=G2}

The preceding section illustrates how one can bootstrap from $v \in C^{2,\alpha}$ to higher regularity, assuming smoothness of the  nonlinear elliptic operator $G_2$.  We now turn our attention to  verifying the assumed smoothness of $G_2$ under the simplifying hypothesis that $G_2=G_1$.  
Our main result is Theorem~\ref{thm: smoothness of G}.
 Though the assumed smoothness of $v$ is addressed for $(n,i)=(1,2)$  in Section \ref{S:ODE},  we consider neither the smoothness nor the uniform convexity 
of $v$ for higher dimensional targets, which as we have noted, remain interesting open questions.

Our joint work with Chiappori \cite{ChiapporiMcCannPass17} establishes regularity of $G_1$ when $n=1$; in this section, we focus on this smoothness for higher dimensional targets.  We note that connectedness of almost every level set $X_2(y,Dv(y),D^2v(y))$, plus the $C^2$-smoothness of $v$ hypothesized in Theorem \ref{thm: bootstrapping}  of the last section, and $C^2$-smoothness of $u=v^s$, implies that $G_1=G_2$ by Theorem \ref{thm: smoothness implies nestedness},   so in many cases of interest
it is enough to address smoothness of $G_1$.
Note however that when  $G_1(y,Dv(y),D^2v(y)) \neq G_2(y,Dv(y),D^2v(y))$, as can happen, for instance, when the $X_2(y,Dv(y),D^2v(y))$ are disconnected, the results in this section 
by themselves yield little information about $G_2$.  


We let  $Y', P' \subseteq \mathbb{R}^n$ be bounded open sets and set $X'=\cup_{(y,p) \in Y' \times P'}X_1(y,p)$.  For technical reasons it is convenient to assume that $y \mapsto s(x,y)$ is uniformly convex
throughout this section; that is, $D^2_{yy}s(x,y) \geq CI>0$ throughout $X' \times Y'$.   Note that this assumption can always be achieved by adding a sufficiently convex function of $y$ to $s$.

\begin{theorem}[Smoothness of $G_1$]\label{thm: smoothness of G}
Let $r \geq 1$.  Then $||G_1||_{C^{r,1}(Y' \times P')}$ is controlled by $||f||_{C^{r,1}(X')}$, $||D_y s||_{C^{r+1,1}(Y' \times X')}$, $||\hat n_ X||_{ C^{r-1,1}}$ and
\begin{eqnarray}
\inf_{(x,y) \in X' \times Y'} \min_{v \in \mathbb{R}^n, |v| =1}| D^2_{xy}s(x,y)\cdot v| &\quad& \mbox{\rm (non-degeneracy),}
\\ \inf_{(x,y,p) \in (\partial X\cap \bar X') \times Y' \times P'} |(\hat n_X)_{T_xX_1(y,p)}| && {\rm (transversality),}
\\ \sup_{(y,p) \in Y' \times P'} \mathcal H^{m-n}(X_1(y,p)) && \mbox{\rm (size of level sets), and of}
\\ \sup_{(y,p) \in Y' \times P'} \mathcal H^{m-n-1}(\overline{X_1(y,p)} \cap \partial X) && 
\mbox{\rm (boundary intersections),}
\end{eqnarray}
  assuming finiteness and positivity of each quantity above.
Here $(\hat n_X)_{T_xX_1(y,p)}$ denotes the projection of the outward unit normal $\hat n_X$ to $X$ onto the tangent space $T_xX_1(y,p)$.
\end{theorem}

Before proving the result, we develop some notation and establish a few preliminary lemmas.

For $i \in \{1,\ldots,n\}$,
the set $X^{i}_{\leq }(y, p): = \{ x \mid s_{y_i} \leq p_i, s_{y_j} =p_j \forall j \neq i\}$
is a submanifold of $X$ whose relative boundary is given by $X_1(y,p)$.
Then $X^{i}_{\leq }(y, p) \subseteq X^{i}(y, p) : =\{ x \mid  s_{y_j} =p_j \forall j \neq i\}$, while with an analogous definition $X^{i}_{= }(y, p) $ coincides with $X_1(y,p)$.

  Nondegeneracy  of $s$ 
  makes $X_1(y,p)$ a codimension one submanifold of the codimension $n-1$ submanifold $X^{i}(y, p)$ of $X$.  By the implicit function theorem, these submanifolds are each one derivative less smooth than $s$.

\begin{lemma}[Submanifold transversality]\label{lemma: transversality of submanifolds}
The submanifolds $\partial X^i = \bar X^i \cap \partial X$ and $X_1$ intersect transversally in $X^i$.
\end{lemma}
\begin{proof}The proof is straightforward linear algebra.
Transversal intersection of $X_1$ and $\partial X$ in $\mathbb{R}^m$ easily implies transversal intersection of $X^i$ and $\partial X$, and so $T_x(\partial X^i) =T_x(\partial X) \cap T_x(X^i)$ at each point of intersection $x \in  \bar X^i \cap \partial X$.  We then need to show

$$
[T_x(\partial X) \cap T_x(X^i)] + T_xX_1 =T_xX^i.
$$ 

The containment $[T_x(\partial X) \cap T_x(X^i)] + T_xX_1 \subseteq T_xX^i$ is immediate, as each of the summands is contained in $T_xX^i$.  On the other hand, if $p \in T_xX^i \subset \mathbb{R}^m =T_x(\partial X) + T_xX^1$ (by transversality), we write $p=p_1+p_\partial$, with $p_1 \in   T_xX^1 \subseteq T_xX^i$ and $p_\partial \in T_x(\partial X)$.  But then $p_\partial =p-p_1 \in T_xX^i$, and so $p_\partial \in  [T_x(\partial X) \cap T_x(X^i)]$, implying the containment $T_xX^i \subseteq [T_x(\partial X) \cap T_x(X^i)] + T_xX_1 $.
\end{proof}

 Given $f \in L^\infty$,
Lemma 5.1 of \cite{ChiapporiMcCannPass17} implies that
$$\Phi^i(y,p):=\int_{X^{i}_{\leq }(y, p)}f(x,y)d\mathcal H ^{m-n+1}(x)$$ 
has a Lipschitz dependence on $p$, with
\begin{equation}\label{eqn: flux differentiation}
\frac{\partial \Phi^i}{\partial p_i}(y,p) =\int_{X^{i}_ =(y, p)}\frac{f(x,y)}{|D_{X^i}s_{y_i}|}d\mathcal H ^{m-n}(x)
\qquad {\rm [a.e.]},
\end{equation}
where $D_{X^i}s_{y_i}$ is the differential of $s_{y_i}$ along the submanifold $X^i$, nonzero by the nondegeneracy assumption:


\begin{lemma}[Restriction non-degeneracy]\label{lemma: nondegeneracy on submanifolds}
The differential $D_{X^i}s_{y_i}$ of $s_{y_i}$ along the manifold $X^i$ satisfies 
$$
|D_{X^i}s_{y_i}| \geq \min_{v \in \mathbf{R}^n,\text{ }|v| =1}| D^2_{xy}s\cdot v|.
$$

\end{lemma}

\begin{proof}
Note that $D_{X^i}s_{y_i}$ is  $D_xs_{y_i}$, minus its projection onto the span of the other $D_xs_{y_j}$, and so

\begin{eqnarray*}
|D_{X^i}s_{y_i}|&=&\min_{v_1,v_2,...v_{i-1},v_{i+1}...v_n}|D_xs_{y_i} -\sum_{j \neq i }v_jD_xs_{y_j}|\\
&=&\min_{v=(v_1,...,v_n) \in \mathbf{R}^n,\text{ }v_i=1} |D^2_{xy}s\cdot v|\\
 &\geq& \min_{v \in \mathbf{R}^n,\text{ }|v|=1}| D^2_{xy}s\cdot v|.
\end{eqnarray*}
\hfill\end{proof}

  Note that the outward unit normal to $X^i_{\leq}(y,p)$ in $X^{i}(y, p)$   is

$$ 
\hat n^i:=\frac{D_{X^i}s_{y_i}}{|D_{X^i}s_{y_i}|}
$$ 
and the normal velocity of $X_1(y,p)$ in $X^{i}(y, p)$ as $p_i$ is varied is 
$$
V^i=\frac{\hat n^i}{|D_{X^i}s_{y_i}|}.
$$
 Here $D_{X^i}s_{y_i}=D_{X^i(y,p)}s_{y_i}(x,y),$ and objects defined in terms of it, such as, $\hat n^i =\hat n^i (x,y,p)$  are defined only for $x \in X^i(y,p)$.  We will denote
$$
D_{X^i}s_{y_i}(x,y):=D_{X^i(y,p)}s_{y_i}(x,y)\Big|_{p=D_ys(x,y)}
$$
which is defined globally on $X' \times Y'$.  Expressions such as $\hat n^i(x,y)$  are defined analogously.  

Similarly, the outward unit normal to $\Big(\overline{X^i_{\leq}(y,p)}\Big) \cap \partial X$ in $\Big(\overline{X^i(y,p)}\Big) \cap \partial X$ will be denoted $\hat n^i_\partial $.  Denote by $\hat n_X^i=\frac{(\hat n_X)_{T_xX^i}}{|(\hat n_X)_{T_xX^i}|}$ the (renormalized) projection of $\hat n_X$ onto $T_xX^i$, which is well-defined by tranversality (note $|(\hat n_X)_{T_xX^i}| \geq |(\hat n_X)_{T_xX_1}|$).  This is the outward unit normal to $\overline X^i(y,p)\cap X$ in $\overline X^i(y,p)$.

We have that 
$$
\hat n^i_\partial =\frac{\hat n^i -(\hat n_X^i\cdot \hat n^i)\hat  n_X^i}{\sqrt{1 -(\hat n_X^i\cdot \hat n^i)^2}}.
$$

 Note that 
 $$
 V^i_\partial :=\frac{ | V^i  |}{\sqrt{1 -(\hat n_X^i\cdot \hat n^i)^2}}\hat n^i_\partial
 $$
 represents the normal velocity of $\Big(\overline{X_1(y,p)}\Big) \cap \partial X$ in $\Big(\overline{X^i(y,p)}\Big) \cap \partial X$.
The denominator is bounded away from $0$ by the transversality assumption.

Analogously to \eqref{eqn: flux differentiation}, Lemma 5.1 in \cite{ChiapporiMcCannPass17} implies for $g \in L^\infty$ that
$\Psi^i(y,p):=\int_{X^{i}_{\leq }(y, p)  \cap \partial X} g(x,y)d\mathcal H ^{m-n}(x)$ has Lipschitz dependence on $p$, 
and

 
\begin{equation}\label{eqn: boundary flux differentiation}
\frac{\partial \Psi^i}{\partial p_i}(y,p) =\int_{\overline {X^{i}_ =(y, p)}\cap \partial X}g(x,y)
{ |  V^i_\partial  |} d{ \mathcal  H} ^{m-n-1}(x)
\qquad {\rm [a.e.].}
\end{equation}

\begin{lemma}[Derivative bounds along submanifolds]
\label{lemma: submanifold differentials}
	Given functions $a:X' \times Y' \rightarrow \mathbb{R}$, $b: \partial X \times Y \rightarrow \mathbb{R}$ and vector fields $v:\bar X'\times Y' \rightarrow TX$ and $w: (\overline{X'} \cap \partial X) \times Y \rightarrow T\partial X$ such that $v(x,y) \in T_xX^i(x,D_ys(x,y))$ and $w(x,y) \in T_x  ( \bar X^i (x,D_ys(x,y)  )\cap \partial X )$ everywhere, we have:

\begin{enumerate}
	\item $||D_{X^i(y,D_ys(x,y))}a(x,y)||_{C^{k,1}(X'\times Y')}$ is controlled by $||a||_{C^{k+1,1}(X'\times Y')}$, $||D_ys||_{C^{k,1}(X'\times Y')}$, and nondegeneracy.
	\item $||\nabla_{X^i(x,D_ys(x,y))} \cdot v||_{C^{k,1}(X'\times Y')}$ is controlled by $||v||_{C^{k+1,1}(X'\times Y')}$.
\item $||D_{\overline{ X^i(y,D_ys(x,y))} \cap \partial X}b(x,y)||_{C^{k,1}((\overline{X'} \cap \partial X)\times Y')}$ is controlled by $||b||_{C^{k+1,1}((\overline{X'} \cap \partial X)\times Y')}$, $||D_ys||_{C^{k,1}((\overline{X'} \cap \partial X)\times Y')}$, nondegeneracy, tranversality and $||\hat n_X||_{C^{k,1}(\overline{X'} \cap \partial X)}$
\item $||\nabla_{\overline{ X^i(x,D_ys(x,y))} \cap \partial X} \cdot w||_{C^{k,1}((\overline{X'} \cap \partial X)\times Y')}$ is controlled by $||w||_{C^{k+1,1}((\overline{X'} \cap \partial X)\times Y')}$ and  $||\hat n_X||_{C^{k+1,1}(\overline{X'} \cap \partial X)}$.
	\end{enumerate}
\end{lemma}

\begin{proof}
First we prove the first implication.  Note that $D_{X^i(y,D_xs(x,y))}a(x,y) $ is equal to  $D_xa(x,y)$, minus it's projection onto the span of the $D_xs_{y_j}$ for $j \neq i$; that is
$$
D_{X^i(y,D_ys(x,y))}a(x,y) = D_xa(x,y) -\sum_{j =1}^{n-1}[ D_xa(x,y)\cdot e_j(x,y)] e_j(x,y)
$$
where the $e_j(x,y)$ are an orthonormal basis for the span of $\{D_xs_{y_j}(x,y)\}_{j \neq i}$.  The $e_j$ can then be written explicitly as functions of the $D_xs_{y_j}(x,y)$, using for instance the Gram-Schmidt procedure; the definition of $e_j$ involves projections onto the $e_{\bar j}$ for $\bar j <j$, which are controlled by nondegeneracy. 

The second implication follows by noting that the divergence $\nabla_{X^i(x,D_ys(x,y))} \cdot v(x,y)$ coincides with $\nabla_{X} \cdot v(x,y)$.

The proof of the third implication is identical to that of the first, except that we subtract the projection onto the span of $\{D_{x}s_{y_j}(x,y)\}_{j \neq i}\cup \{\hat n_X\}$;
 this is controlled by nondegeneracy and transversality, as well as the smoothness of these basis vectors.

Finally, the proof of the fourth assertion is almost the same as the second; the divergence coincides with $\nabla_{\partial X} \cdot w(x,y)$,  which involves first derivatives of the metric, and hence of $\hat n_X$, as in the remarks preceding Lemma 7.2 in \cite{ChiapporiMcCannPass17}.
\end{proof}

Now, we define $s^*(x,p)$ to be the Legendre transformation of $s$ with respect to the $y$ variable:
$$
s^*(x,p) =\sup_y(y \cdot p - s(x,y)).
$$

\begin{lemma}[Smoothness and non-degeneracy for Legendre duals]\label{lemma s* smoothness}
The transformation $s^*$ inherits the same smoothness as $s$, and is non-degenerate.  Further, its non-degeneracy is quantitatively controlled by the non-degeneracy and $C^2$ norm of $s$:

$$
\inf_{|u|=1}| D^2_{xp}s^*(x,p)\cdot u| \geq \frac{\inf_{|v|=1}|D^2_{xy}s(x,y)\cdot v|}{||D^2_{yy}s(x,y)||}
$$
for $p=D_ys(x,y)$.
\end{lemma}

\begin{proof}
Uniform convexity implies that $s^*$ is continuously twice differentiable with respect to $p$.  The implicit function theorem combined with the identity $D_ps^*(x,D_ys(x,y)) =y$ implies the smoothness of $s^*$.  In particular, differentiating with respect to $x$ yields

$$
D^2_{xp}s^*(x,D_ys(x,y)) =-D^2_{xy}s(x,y) D^2_{pp}s^*(x,D_ys(x,y))
$$
and so invertibility of $D^2_{pp}s^*$ and nondegeneracy of $s$ imply nondegeneracy of $s^*$, and we have, for $|u|=1$,

\begin{eqnarray*}
D^2_{xp}s^*(x,D_ys(x,y))\cdot u& =&-D^2_{xy}s(x,y) D^2_{pp}s^*(x,D_ys(x,y))\cdot u\\
&=&-D^2_{xy}s(x,y)\frac{ D^2_{pp}s^*(x,D_ys(x,y))\cdot u}{|D^2_{pp}s^*(x,D_ys(x,y))\cdot u|}|D^2_{pp}s^*(x,D_ys(x,y))\cdot u|.
\end{eqnarray*}
Now note that setting $v=D^2_{pp}s^*(x,D_ys(x,y))\cdot u =[D^2_{yy}s(x,y)]^{-1}\cdot u$, so that $1=|u| =|D^2_{yy}s(x,y)\cdot v| \leq||D^2_{yy}s(x,y)||\cdot| v|$.  Therefore
$$
|v| \geq \frac{1}{||D^2_{yy}s(x,y)||}
$$ 
 and the result follows.
\end{proof}

Now, we can identify the set $X_1(y,p)=\{ x \mid D_ps^*(x,p)=y\}$.  We then define  $X^{*i}_{\leq }(y, p)$, $ X^{*i}(y, p)$ and $\Phi^{*i}$  analogously to above, and compute 

\begin{equation}
\frac{\partial \Phi^{*i}}{\partial y_i} =\int_{X^{*i}_ =(y, p)}\frac{f(x,y)}{|D_{X^{*i}}s^*_{p_i}|}d\mathcal H ^{m-n-1}(x) +\int_{X^{*i}_{\leq }(y, p)}\frac{\partial f}{\partial y_i} (x,y)d\mathcal H ^{m-n}(x) 
\end{equation}
for a.e. $(y,p)$ as long as $f$ and $f_{y_i}$ are Lipschitz.

Analogs of Lemmas \ref{lemma: transversality of submanifolds}, \ref{lemma: nondegeneracy on submanifolds} and \ref{lemma:  submanifold differentials} when $s(x,y)$ is replaced by $s^*(x,p)$ then follow immediately.  We note that

$$
D_{X^{i*}}s^*_{p_i}(x,y):=D_{X^{i^*}(y,p)}s^*_{p_i}(x,p)\Big|_{p=D_ys(x,y)}
$$
is defined throughout $X' \times Y'$.  We define $\hat n^{*i}$, $V^{*i}$, $\hat n_\partial ^{*i}$, $\hat n^{*i}_X$, $V^{*i}$ analogously to their un-starred counterparts and note that upon evaluating at $p=D_ys(x,y)$, each can be considered a function on $X' \times Y'$ or $\partial X ' \times Y'$.


\begin{lemma}[Flux derivatives through moving surfaces]
\label{lemma: differentiating level set integrals}
 Use $a: \overline {X' \times Y' \times P'} \rightarrow \mathbb{R}$ Lipschitz 
 to define $\Phi(y,p) := \int_{X_1(y,p)}a(x,y,p)d\mathcal{H}^{m-n}(x)$ and $\Psi(y,p): = \int_{\overline{X_1(y,p)} \cap \partial X}a(x,y,p)d\mathcal{H}^{m-n-1}(x)$.  Then $\Phi$ and $\Psi$ are Lipschitz with  partial derivatives given almost everywhere by:

\begin{eqnarray}\nonumber
\frac{\partial \Phi(y,p)}{\partial p_i}& =&  \int_{X_1(y,p)}\Big[\nabla_{X^{i}(y, p)} \cdot \Big( a(x,y,p)\frac{D_{X^i}s_{y_i}}{|D_{X^i}s_{y_i}|}\Big) V^i \cdot  \hat n^i\Big]_{p=D_ys(x,y)} d\mathcal{H}^{m-n}(x)\\ \nonumber
&-&\int_{\Big(\overline{X_1(y,p)}\Big) \cap \partial X} \Big[ \Big( a(x,y,p)\frac{D_{X^i}s_{y_i}}{|D_{X^i}s_{y_i}|}\Big)\cdot \hat n_X^iV^i_\partial \cdot \hat n^i_\partial \Big]_{p=D_ys(x,y)} d\mathcal{H}^{m-n-1}(x)\\
&+& \int_{X_1(y,p)}\Big[\frac{\partial a(x,y,p)}{\partial p^i}\Big]_{p=D_ys(x,y)} d\mathcal{H}^{m-n}(x),\label{eqn: interior p derivative}
\end{eqnarray}

\begin{align}\nonumber
\frac{\partial \Psi(y,p)}{\partial p_i} &=\\&\int_{\overline{X_1(y,p)} \cap \partial X}\Big[\nabla_{\overline{X^{i}(y, p)} \cap \partial X}\cdot \Big( a(x,y,p)\frac{D_{\overline{X_i(y,p)} \cap \partial X}s_{y_i}}{|D_{\overline{X_i(y,p)} \cap \partial X}s_{y_i}|}\Big) V^i_{\partial} \cdot  \hat n^{i}_{\partial} \Big]_{p=D_ys(x,y)} d\mathcal{H}^{m-n}(x)\nonumber\\
&+ \int_{\overline{X_1(y,p)} \cap \partial X}\Big[\frac{\partial a(x,y,p)}{\partial p^i}\Big]_{p=D_ys(x,y)} d\mathcal{H}^{m-n-1}(x),\label{eqn: boundary p derivative}
\end{align}

\begin{eqnarray}\nonumber
\frac{\partial \Phi(y,p)}{\partial y_i}& =&  \int_{X_1(y,p)}\Big[\nabla_{X^{*i}(y, p)}\cdot  \Big( a(x,y,p)\frac{D_{X^{*i}}s^*_{p_i}}{|D_{X^{*i}}s^*_{p_i}|}\Big) V^{*i} \cdot  \hat n^{*i} \Big]_{p=D_ys(x,y)} d\mathcal{H}^{m-n}(x)\\\nonumber
&-&\int_{\Big(\overline{X_1(y,p)}\Big) \cap \partial X}\Big[  \Big( a(x,y,p)\frac{D_{X^{*i}}s^*_{p_i}}{|D_{X^{*i}}s_{p_i}|}\Big)\cdot \hat n_X^{*i}V^{*i}_{\partial} \cdot \hat n^{*i}_\partial\Big]_{p=D_ys(x,y)}  d\mathcal{H}^{m-n-1}(x)\\
&+& \int_{X_1(y,p)}\Big[\frac{\partial a(x,y,p)}{\partial y^i}\Big]_{p=D_ys(x,y)} d\mathcal{H}^{m-n}(x),\label{eqn: interior y derivative}
\end{eqnarray}
 
and
\begin{align}\nonumber
\frac{\partial \Psi(y,p)}{\partial y_i}& =\\ &\int_{\overline{X_1(y,p)} \cap \partial X}\Big[\nabla_{\overline{X^{*i}(y, p)}\cap \partial X} \cdot \Big( a(x,y,p)\frac{D_{\overline{X^{*i}(y,p)} \cap \partial X}s^*_{p_i}}{|D_{\overline{X^{*i}(y,p)} \cap \partial X}s^*_{p_i}|}\Big) V^{*i}_{\partial} \cdot  \hat n^{*i\partial} \Big]_{p=D_ys(x,y)} d\mathcal{H}^{m-n}(x)\nonumber\\
&+ \int_{\overline{X_1(y,p)} \cap \partial X}\Big[\frac{\partial a(x,y,p)}{\partial y^i}\Big]_{p=D_ys(x,y)} d\mathcal{H}^{m-n-1}(x).\label{eqn: boundary y derivative}
\end{align}
\end{lemma}


\begin{proof}
We begin by establishing the formulas assuming $a \in C^{1,1}(\overline{X' \times Y' \times P'})$.

Using the generalized divergence theorem \cite[Proposition 27]{HofmannMitreaTaylor10}

\begin{eqnarray*}
\Phi(y,p)& =& \int_{X_1(y,p)}\Big( a(x,y,p)\frac{D_{X^i}s_{y_i}}{|D_{X^i}s_{y_i}|}\Big) \cdot \hat n^id\mathcal{H}^{m-n}(x)\\
&=& \int_{X^i_{\leq}(y,p) \setminus X^i_{\leq}(y,p^{(i)})}\nabla_{X^{i}(y, p)}\cdot \Big( a(x,y,p)\frac{D_{X^i}s_{y_i}}{|D_{X^i}s_{y_i}|}\Big) d\mathcal{H}^{m-n+1}(x)  \\
&+&\int_{X^i_{=}(y,p^{(i)})}\Big( a(x,y,p)\frac{D_{X^i}s_{y_i}}{|D_{X^i}s_{y_i}|}\Big)\cdot \hat n^i(x,y)d\mathcal{H}^{m-n}(x)\\
&-&\int_{\Big(X^i_{\leq}(y,p) \setminus X^i_{\leq}(y,p^{(i)})\Big) \cap \partial X}  \Big( a(x,y,p)\frac{D_{X^i}s_{y_i}}{|D_{X^i}s_{y_i}|}\Big)\cdot \hat n_X^i(x,y)d\mathcal{H}^{m-n}(x).
\end{eqnarray*}
Noting that the integrands in the first and third terms above are bounded, one can then combine the chain rule with \eqref{eqn: flux differentiation} and  \eqref{eqn: boundary flux differentiation} to differentiate with respect to $p_i$, getting 
\begin{eqnarray*}
\frac{\partial \Phi(y,p)}{\partial p_i}& =&  \int_{X^i_{=}(y,p)}\nabla_{X^{i}(y, p)} \cdot\Big( a(x,y,p)\frac{D_{X^i}s_{y_i}}{|D_{X^i}s_{y_i}|}\Big) V^i \cdot  \hat n^i d\mathcal{H}^{m-n}(x)\\
&-&\int_{\Big(\overline{X^i_=(y,p)}\Big) \cap \partial X}  \Big( a(x,y,p)\frac{D_{X^i}s_{y_i}}{|D_{X^i}s_{y_i}|}\Big)\cdot \hat n_X^iV^i_\partial \cdot \hat n^i_\partial d\mathcal{H}^{m-n-1}(x)\\
&+& \int_{X_1(y,p)}\frac{\partial a(x,y,p)}{\partial p^i}d\mathcal{H}^{m-n}(x).
\end{eqnarray*}
Finally, notice that one may substitute $p=D_ys(x,y)$ in each integrand, as each region of integration is contained in $\overline X_1(y,p) $, to establish \eqref{eqn: interior p derivative} for $a \in C^{1,1}$.

Now, note that the formula \eqref{eqn: interior p derivative} for $\frac{\partial \Phi(y,p)}{\partial p_i}$ is controlled by $||a||_{C^{0,1}}$ (that is, it does not depend on $||a||_{C^{1,1}}$).  For $a$ merely Lipschitz, we can therefore choose a sequence $a_n \in C^{1,1}$ converging to $a$ in the $C^{0,1}$ norm; passing to the limit implies that $||\Phi||_{C^{0,1}(\overline {Y' \times P'})}$ is controlled by $||a||_{C^{0,1}}$, and, using the dominated convergence theorem, one obtains the desired formula.

A similar argument applies to the boundary integral terms to produce the desired formula \eqref{eqn: boundary p derivative} for $\frac{\partial \Psi(y,p)}{\partial p_i}$, while essentially identical arguments apply to the $y$ derivatives, yielding \eqref{eqn: interior y derivative} and \eqref{eqn: boundary y derivative}.
%
%
\end{proof}

\begin{corollary}[Iterated derivative bounds]\label{lemma: iterating derivatives}
	The operators 
	$$
	A_{p_i}: (a,b)	\mapsto  (a_p^i,b_p^i) \quad \text{and} \quad A_{y_i}: (a,b)	\mapsto  (a_y^i,b_y^i),
	$$
given by
	\begin{eqnarray*}
a_p^i &:=&\Big[\nabla_{X^{i}(y, p)}\cdot \Big( a(x,y)\frac{D_{X^i}s_{y_i}}{|D_{X^i}s_{y_i}|}\Big) V^i \cdot  \hat n^i \Big]_{p=D_ys(x,y)},
\\
b_p^i&:=&\Big[\Big( a(x,y)\frac{D_{X^i}s_{y_i}}{|D_{X^i}s_{y_i}|}\Big)\cdot \hat n_X^iV^i_\partial \cdot \hat n^i_\partial
\\ &&+\nabla_{\overline{X^{i}(y, p)} \cap \partial X} \cdot \Big( b(x,y)\frac{D_{\overline{X_i(y,p)} \cap \partial X}s_{y_i}}{|D_{\overline{X_i(y,p)} \cap \partial X}s_{y_i}|}\Big) V^i_{\partial} \cdot  \hat n^{i}_{\partial}  \Big]_{p=D_ys(x,y)},
%
\\a_y^i &:=& \Big[\nabla_{X^{*i}(y, p)}\cdot \Big( a(x,y)\frac{D_{X^{*i}}s^*_{p_i}}{|D_{X^{*i}}s^*_{p_i}|}\Big) V^{*i} \cdot  \hat n^{*i} +\frac{\partial a(x,y)}{\partial y^{i}}\Big]_{p=D_ys(x,y)},
{ and} 
\\ b_y^i&:=&\Big[\frac{\partial b(x,y)}{\partial y^i} + \Big(  a(x,y)\frac{D_{X^{*i}}s^*_{p_i}}{|D_{X^{*i}}s^*_{p_i}|}\Big)\cdot \hat n_X^{*i}V^{*i}_\partial \cdot \hat n^{*i}_\partial\\
&&+ \nabla_{\overline{X^{*i}(y, p)}\cap \partial X} \cdot \Big(  b(x,y)\frac{D_{\overline{X^{*i}(y,p)} \cap \partial X}s^*_{p_i}}{|D_{\overline{X^{*i}(y,p)} \cap \partial X}s^*_{p_i}|}\Big) V^{*i}_{\partial} \cdot  \hat n^{*i\partial}  \Big]_{p=D_ys(x,y)},
\end{eqnarray*}
define mappings  $A_{p_i}:B_k \rightarrow B_{k-1}$ and
$A_{y_i}:B_k \rightarrow B_{k-1}$ between Banach spaces defined by
$$
B_k:=C^{k,1}(X' \times Y' ) \oplus C^{k,1}([X'\cap \partial X] \times Y' )
$$
with norms 
\begin{eqnarray*}
||A_{p_i}||& \leq& ||\frac{1}{|D_{X^i}s_{y_i}|}||_{{C^{k-1,1}(X' \times Y')}}||\hat n^i||_{C^{k-1,1}(X' \times Y')}\\
&+&||\hat n^i||_{{C^{k-1,1}(X' \times Y')}}  ||\hat n_X^i||_{{C^{k-1,1}( (\overline{X'} \cap \partial X)\times Y')}}||V^i_\partial \cdot \hat n^i_\partial||_{{C^{k-1,1}}( (\overline{X'} \cap \partial X) \times Y')}  \\
&+&
 ||\frac{D_{\overline{X^i(y,p)} \cap \partial X}s_{y_i}}{|D_{\overline{X^i(y,p)} \cap \partial X}s_{y_i}|}||_{C^{k-1,1}( (\overline{X'} \cap \partial X) \times Y')} ||V^i_{\partial}||_{{C^{k-1,1}( (\overline{X'} \cap \partial X) \times Y')}}||  \hat n^{i}_{\partial}||_{{C^{k-1,1}( (\overline{X'} \cap \partial X) \times Y')}} 
\end{eqnarray*}
and
\begin{align*}
	||A_{y_i}|| &\leq ||\frac{1}{|D_{X^{*i}}s^*_{p_i}|}||_{{C^{k-1,1}( X' \times Y')}}||\hat n^{*i}||_{C^{k-1,1}( X' \times Y')}+1\\
	+& ||\hat n^{*i}||_{{C^{k-1,1}( X' \times Y')}}  ||\hat n_X^{*i}||_{{C^{k-1,1}( (\overline{X'} \cap \partial X) \times Y')}}||V^{*i}_\partial \cdot \hat n^{*i}_\partial||_{{C^{k-1,1}}( (\overline{X'} \cap \partial X)\times Y')}  \\
	+&
	||\frac{D_{\overline{X^{*i}(y,p)} \cap \partial X}s^*_{p_i}}{|D_{\overline{X^{*i}(y,p)} \cap \partial X}s^*_{p_i}|}||_{C^{k-1,1}( (\overline{X'} \cap \partial X) \times Y')} ||V^{*i}_{\partial}||_{{C^{k-1,1}( (\overline{X'} \cap \partial X) \times Y')}}||  \hat n^{*i}_{\partial}||_{{C^{k-1,1}( (\overline{X'} \cap \partial X)' \times Y')}} 
\end{align*}
controlled by	$ ||D_ys||_{C^{k,1}}$, $||\hat n_X||_{C^{ k,1}}$, non-degeneracy and transversality.

	Furthermore, restricted to the subspace $C^{k,1}(X' \times Y' ) \oplus\{0\}$, the norms
	
\begin{align*}	
||A_{p_i}||&_{C^{k,1}(X' \times Y' ) \oplus\{0\} \rightarrow B_{k-1}}\leq ||\frac{1}{|D_{X^i}s_{y_i}|}||_{{C^{k-1,1}(X' \times Y')}}||\hat n^i||_{C^{k-1,1}(X' \times Y')}\\
&+||\hat n^i||_{{C^{k-1,1}(X' \times Y')}}  ||\hat n_X^i||_{{C^{k-1,1}((\overline{X'} \cap \partial X) \times Y')}}||V^i_\partial \cdot \hat n^i_\partial||_{{C^{k-1,1}}((\overline{X'} \cap \partial X) \times Y')}  
\end{align*}
and
\begin{align*}
||A_{y_i}||&_{C^{k,1}(X' \times Y' ) \oplus\{0\} \rightarrow B_{k-1}} \leq ||\frac{1}{|D_{X^{*i}}s^*_{p_i}|}||_{{C^{k-1,1}( X' \times Y')}}||\hat n^{*i}||_{C^{k-1,1}( X' \times Y')}+1\\
&+ ||\hat n^{*i}||_{{C^{k-1,1}( X' \times Y')}}  ||\hat n_X^{*i}||_{{C^{k-1,1}((\overline{X'} \cap \partial X) \times Y')}}||V^{*i}_\partial \cdot \hat n^{*i}_\partial||_{{C^{k-1,1}}((\overline{X'} \cap \partial X) \times Y')}   
\end{align*}
are controlled by $||D_ys||_{C^{k,1}}$, $||\hat n_X||_{C^{ k-1,1}}$, non-degeneracy and transversality.
\end{corollary}
\begin{proof}
The estimates on the norms follow by simple calculations.  The control on the various quantities in the estimates relies on Lemmas \ref{lemma: nondegeneracy on submanifolds}, \ref{lemma: submanifold differentials}, \ref{lemma s* smoothness}, and closure of the H\"older spaces $C^{k-1,1}$ under composition.
\end{proof}




 We now prove the result announced at the beginning of this section:

\begin{proof}[Proof Theorem \ref{thm: smoothness of G}] 
First note that as $Q$ enters the definition of of $G_1$ only through the  integrand, whose 
dependence on $Q$ is smooth, computing derivatives with respect to $Q$ is straightforward.
	
Corollary \ref{lemma: iterating derivatives} allows us to iterate derivatives with respect to the other variables; given multi indices $\alpha =(\alpha_1,\alpha_2,...\alpha_n)$, $\beta=(\beta_1,\beta_2,...,\beta_n)$, and $\gamma=(\gamma_1,\gamma_2,....,\gamma_{n^2})$ with $|\alpha|+|\beta|+|\gamma|	=\sum_{i=1}^n \alpha_i +\sum_{i=1}^n \beta_i+\sum_{i=1}^{n^2} \gamma_i  =  k \leq r$, then  Lemma \ref{lemma: differentiating level set integrals} and Corollary \ref{lemma: iterating derivatives} allow us to compute 
\begin{equation}\label{eqn: G_1 derivatives}
\frac{\partial^kG_1}{\partial p^\alpha \partial y^\beta\partial Q^\gamma} =\int_{X_1(y,p)}a^{\alpha,\beta} d\mathcal{H}^{m-n} +\int_{\partial \bar X_1(y,p) \cap \partial X_1}b^{\alpha,\beta} d\mathcal{H}^{m-n-1} 
\end{equation}
where $(a^{\alpha,\beta},b^{\alpha,\beta}) = A^{\alpha}A^{\beta}(\frac{\partial^{|\gamma|} h}{\partial Q^\gamma },0) \in B_{r-k}$, with $h(x,y,p) = \frac{\det [Q-D^2_{yy} s(x,y)]}{\sqrt{\det D^2_{xy} s(x,y) (D^2_{xy} s(x,y))^T} } 
f(x) $ being the original integrand in the definition of $G_1(y,p,Q)$, and $A^{\alpha} =A_{p_1}^{\alpha_1}....A_{p_n}^{\alpha_n}$, $A^{\beta} =A_{y_1}^{\beta_1}....A_{y_n}^{\beta_n}$. %
Now, Corollary \ref{lemma: iterating derivatives} implies that $||(a^{\alpha,\beta},b^{\alpha,\beta})||_{C^{r-k,1}}$ is controlled by $||f||_{C^{r,1}},  ||D_ys||_{C^{r+1,1}}$,  $||\hat n_X||_{C^{r-1,1}}$, non-degeneracy and transversality.

It then follows from \eqref{eqn: G_1 derivatives} that $\frac{\partial ^kG_1}{\partial p^\alpha \partial y^\beta\partial Q^\gamma}$ is controlled by the quantities  listed in the statement of the present theorem for
 $k \leq r$, as desired.
\end{proof}

\section{Smoothness of the local operator $G_2$ for one dimensional targets}
\label{S:G2}

Taken together, the two preceding sections allow one to bootstrap from $C^{2,\alpha}$ to higher regularity, when $X_2 =X_1$.  This raises  the following natural questions:

\begin{enumerate}
	\item When $X_2$ and $X_1$ differ (in which case the results in the previous subsection do not tell us much about solutions of the $i=2$ equation), under what conditions is the elliptic operator $G_2$ smooth?
	\item When can we confirm solutions are $C^{2,\alpha}$, allowing one to apply Theorem \ref{thm: bootstrapping}?
\end{enumerate}

The goal of this section and the next is to fill these gaps for one dimensional targets, $n=1$.  In this section, we identify conditions under which $G_2$ is smooth.  As in the previous section, where regularity of $G_1$ for higher dimensional targets was considered, the general strategy is to adapt the approach in \cite{ChiapporiMcCannPass17}, using the divergence theorem to convert integrals over regions to those over boundaries, and differentiating the latter using the calculus of moving boundaries.  These results, combined with general ODE theory, imply that $C^{1,1}_{loc}$ solutions to the $i=2$ equation are in fact $C^{2,1}_{loc}$; higher order regularity estimates on $G_2$ in turn yield higher order regularity of these solutions.

The second question above is deferred to  Section \ref{S:ODE}, 
where we find conditions under which any almost everywhere solution to the $i=2$ equation with the one dimensional targets is locally $C^{1,1}$; the results of the present section then imply that these solutions are smoother, depending on the degree of regularity of $G_2$.

 
Given regions $Y'$, $P'$ and $Q'$ in $\mathbf{R}$, we set $X' =\cup_{(y,p, q) \in Y' \times P'\times Q'} X_{2}(y,p,q)$
Assume $D_x s_y$ and $D_x s_{yy}$ are linearly independent throughout $X' \times Y' \subseteq X \times Y$.

As $p$ is increased,
the domain $W_\le (y,p) := \{x \in X \mid s_y \le p \}$ expands monotonically outward with 
normal velocity $w(x,y):=|D_x s_y|^{-1}$ along its interface 
$W_= = X_1$.   Its normal velocity with respect to changes in $y$ is $-s_{yy} w$. 
Similarly, as $q$ is increased $Z_\le (y,q) := \{x\in X \mid s_{yy} \le q\}$
expands monotonically outward with normal velocity $z(x,y) := |D_x s_{yy}|^{-1}$ along its 
interface $Z_=(y,q) := \{x \in X \mid s_{yy}=q\}$; its  normal velocity with respect to changes in $y$ is $-s_{yyy} z$.  Our linear independence assumption
guarantees these velocities are finite and $W_=$ intersects $Z_=$ transversally.
Notice $X_2(y,p,q) = W_=(y,p) \cap Z_\le(y,q)$.  Assume also, in the same region of interest,
that both $\overline{W_=\cap Z_\le}$ and $\overline{W_= \cap Z_=}$ intersect $\p X$ transversally.
We denote by $\hat n_W = w D_x s_y$ and $\hat n_Z =z D_x s_{yy}$ the outer normals to 
$W_\le$ and $Z_\le$ respectively, and observe that the frontier of e.g. $W_\le$ moves
with velocity $w/\sin \theta$ in $Z_=$,  when $\hat n_Z \cdot \hat n_W = \cos \theta$.

Our main theorem of this section is the following.


\begin{theorem}[Smoothness of the ODE given by $G_2$]
\label{thm: smoothness of G_2}
 If $n=1$ and $r \ge 0$ is an integer, then
	$||G_2||_{C^{r,1}(Y' \times P' \times Q')}$ is controlled by $||f||_{C^{r,1}(X')}$, $||s_y||_{C^{r+2,1}(Y' \times X')}$, $||\hat n_ X||_{ C^{r-1,1}}$ and
	\begin{eqnarray}
	\inf_{(x,y) \in X' \times Y'} \min\{|D_xs_y(x,y)|,|D_xs_{yy}(x,y)| \} &\quad& \mbox{\rm (non-degeneracy),} \\
	\inf_{(x,y) \in X' \times Y'} 1-(\hat n_W \cdot \hat n_Z)^2 &\quad& \mbox{\rm ( transversality),}
	\\ \inf_{(x,y) \in (\partial X\cap \bar X') \times Y',\text{ } \lambda_1^2+\lambda_2^2 +\lambda_3^2 =1} |\lambda_1\hat n_W +\lambda_2\hat n_Z +\lambda_3\hat n_X| && {\rm ( linear\ independence),}\label{eqn: linear independence}
	\\ \sup_{(y,p,q) \in Y' \times P' \times Q'} \mathcal H^{m-1}(W_=(y,p) \cap Z_\le(y,q)  )
&& \mbox{\rm (1st level set size),}
	\\ \sup_{(y,p,q) \in Y' \times P' \times Q'} \mathcal H^{m-1}(W_\le(y,p) \cap Z_=(y,q)  )
&& \mbox{\rm (2nd level set size),}
	\\ \sup_{(y,p,q) \in Y' \times P' \times Q'} \mathcal H^{m}(W_\le(y,p) \cap Z_\le(y,q)  )
&& \mbox{\rm (iterated sublevel size),}
	\\ \sup_{(y,p,q) \in Y' \times P' \times Q'} \mathcal H^{m-2}((\overline{W_=(y,p) \cap Z_\le(y,q)}) \cap \partial X) && 
	\mbox{\rm (1st boundary level size),}
	\\ \sup_{(y,p,q) \in Y' \times P' \times Q'} \mathcal H^{m-2}((\overline{W_\le(y,p) \cap Z_=(y,q)}) \cap \partial X) && 
	\mbox{\rm (2nd boundary level size), and}
	\\ \sup_{(y,p,q) \in Y' \times P' \times Q'} \mathcal H^{m-1}((\overline{W_\le(y,p) \cap Z_\le(y,q)}) \cap \partial X) && 
	\mbox{\rm (boundary sublevels size),}
	\end{eqnarray}
	assuming finiteness and positivity of each quantity above.
\end{theorem}
 \begin{remark}
		When $m=2$, the linear independence assumption \eqref{eqn: linear independence} cannot hold.  It can be replaced by the assumption that $\{x \in \bar X: (s_y(x,y), s_{yy}(x,y)) \in P' \times Q'\}$ does not intersect $\partial X$ for any $y \in Y'$; in this case, quantity  \eqref{eqn: linear independence} should be replaced by the pairwise linear independence quantities
		$$
			\inf_{(x,y) \in (\bar X'\cap\partial X) \times Y'} 1-(\hat n_W \cdot \hat n_X)^2
		$$
		and 
			$$
		\inf_{(x,y) \in (\bar X'\cap\partial X) \times Y'} 1-(\hat n_X \cdot \hat n_Z)^2.
		$$
\end{remark}

\begin{lemma}[Derivatives on moving submanifolds-with-boundary]
\label{lemma: derivatives of X_2 integrals}
	Given real-valued Lipschitz functions $a,b,c$ on $X' \times Y' \times P' \times Q',$ and $a^\partial,b^\partial, c^\partial$ on  $\p X' \times Y' \times P' \times Q',$ the functions
	\begin{eqnarray*}
	A(y,p,q)&:= &\int_{X_2(y,p,q)} a(x,y,p,q) d{\mathcal H}^{m-1}(x),\\
	B(y,p,q)&:= &\int_{W_\leq(y,p) \cap Z_\leq(y,q)} b(x,y,p,q) d{\mathcal H}^{m}(x),\\
		C(y,p,q)&:=& \int_{W_\le(y,p) \cap  Z_=(y,q)} c(x,y,p,q) d{\mathcal H}^{m-1}(x),\\
	A^\p(y,p,q)&:= &\int_{(\overline{W_=(y,p)\cap  Z_\leq(y,q)})\cap \partial X} a^\p(x,y,p,q) d{\mathcal H}^{m-2}(x),\\
	B^\p(y,p,q)&:= &\int_{(\overline{W_\le(y,p)\cap  Z_\leq(y,q)})\cap \partial X} b^\p(x,y,p,q) d{\mathcal H}^{m-1}(x)\text{ and}\\
	C^\p(y,p,q)&:= &\int_{(\overline{W_\le(y,p)\cap  Z_=(y,q)})\cap \partial X} c^\p(x,y,p,q) d{\mathcal H}^{m-2}(x)
	\end{eqnarray*}
are all Lipschitz, with derivatives given almost everywhere  by
the formulae in Appendix \ref{appendix: derivatives}.  Here  
$w:=|D_x s_y|^{-1}$ and $z:= |D_x s_{yy}|^{-1}$ as above.
\end{lemma}

\begin{proof}
	The proof is similar to the proofs of Lemma 7.4 in \cite{ChiapporiMcCannPass17} and Lemma \ref{lemma: differentiating level set integrals} in the present paper; we only described the main differences here.  For a sufficiently smooth integrand, the derivative of $A$ with respect to $p$, for example, includes a term capturing differentiation of the integrand with respect to $a$, and a term capturing the dependence of the region of integration, which we compute using the generalized divergence theorem and Lemma 5.1 in \cite{ChiapporiMcCannPass17}:
	
	\begin{eqnarray*}
		A_p - &&\int_{X_2} a_p d{\mathcal H}^{m-1} = \frac{\p}{\p \tilde p}\bigg|_{\tilde p= p} \int_{W_=(y,\tilde p) \cap Z_\le(y,q)} 
		a(x,y,p,q) \hat n_W \cdot \hat n_W d{\mathcal H}^{m-1}(x)
		\\ &=&  \frac{\p}{\p \tilde p}\bigg[ \int_{W_\le \cap Z_\le} \nabla \cdot (a\hat n_W)  d{\mathcal H}^{m}-
		\int_{W_\le \cap Z_=} a \hat n_W \cdot \hat n_Z d{\mathcal H}^{m-1}
		- \int_{\overline{{W_\le \cap Z_\le}} \cap \p X} a \hat n_W \cdot \hat n_X d{\mathcal H}^{m-1}
		\bigg]_{\tilde p=p}
		\\ &=& \int_{W_= \cap Z_\le} \nabla \cdot (a\hat n_W) w d{\mathcal H}^{m-1}
		- \int_{W_=\cap Z_=}  \frac{aw \hat n_W \cdot \hat n_Z}{\sqrt{1 - (\hat n_W \cdot \hat n_Z)^2}} d{\mathcal H}^{m-2}
		\\ && - \int_{\overline{{W_= \cap Z_\le}} \cap \p X} \frac{aw \hat n_W \cdot \hat n_X}{\sqrt{1 - (\hat n_W \cdot \hat n_X)^2}}d{\mathcal H}^{m-2}.
	\end{eqnarray*}
The result for Lipschitz functions can be obtained as in Lemma 7.4 in \cite{ChiapporiMcCannPass17} and Lemma \ref{lemma: differentiating level set integrals} here, via approximation by $C^{1,1}$ integrands and the dominated convergence theorem.  The arguments for other derivatives of $A$, $B$ and $C$ are similar.
	We treat boundary integrals analogously.  Noting that, for instance, 
\begin{eqnarray*}
	A^\partial(y,p,q) &=&\int_{(\overline{W_= \cap Z_{\leq})}\cap \partial X)} a^\p(x,y,p,q) d{\mathcal H}^{m-2}(x)\\
	&=&\int_{(\overline{W_\leq \cap Z_{\leq})}\cap \partial X)} \nabla_{\partial X} \cdot (a^\p \hat n_{\partial, W})d{\mathcal H}^{m-1}(x) -\int_{(\overline{W_\le \cap Z_{=})}\cap \partial X)} a^\p\hat n_{\partial, W}\cdot \hat n_{\partial, Z} d{\mathcal H}^{m-2}(x),
\end{eqnarray*}
where  $\hat n_{\partial, W}$ and $\hat n_{\partial, Z}$ are defined in Appendix \ref{appendix: derivatives}, we can again use Lemma 5.1 in \cite{ChiapporiMcCannPass17} to differentiate with respect to $p$.  Similar arguments apply to all derivatives of $A^\p$, $B^\p$ and $C^\p$.
\end{proof}

We note that all integrals over the domain $W_=\cap Z_=$ can be rewritten using the divergence theorem as follows:
\begin{eqnarray*}
\int_{W_=\cap Z_=}a(x,y,p,q)d\mathcal H^{m-2}(x)& =&\int_{W_= \cap Z_\le}\nabla_{W_=} \cdot(\alpha \hat n_{{W_=, Z}})d\mathcal H^{m-1}(x) \\
&-& \int_{(\overline{W_= \cap Z_\le})\cap \partial X}a \hat n_{W_=,Z}\cdot \hat n_{W_=,X}d\mathcal H^{m-2}(x)
\end{eqnarray*}
where  $\hat n_{W_=, Z} := \frac{\hat n_Z -(\hat n_Z \cdot \hat n_W)\hat n_W }{\sqrt{1-(\hat n_Z \cdot \hat n_W)^2}}$ and $\hat n_{W_=, X} := \frac{\hat n_X -(\hat n_X \cdot \hat n_W)\hat n_W }{\sqrt{1-(\hat n_X \cdot \hat n_W)^2}}$ are the outward unit normals in the submanifold $W_= \subseteq X$ to $Z_\le$ and $X$, respectively.

Similarly, 

\begin{eqnarray*}
	\int_{(\overline{W_=\cap Z_=})\cap \partial X}a(x,y,p,q)d\mathcal H^{m-3}(x)& =&\int_{(\overline{W_=\cap Z_\le})\cap \partial X}\nabla_{\bar W_= \cap \partial X} \cdot(a\hat n_{{\bar W_= \cap \partial X, Z}})d\mathcal H^{m- 2}(x),
\end{eqnarray*}
where $\hat n_{{\bar W_= \cap \partial X, Z}}$ is the outward unit normal to $\bar Z_\leq \cap(\bar W_= \cap \partial X)$ in the codimension 2 submanifold $(\bar W_= \cap \partial X)$; alternatively, it is equal to $\hat n_Z$, minus its projection onto the span of $\hat n_{X}$ and $\hat n_W$.
This means that differentiating a function of any of the types in Lemma \ref{lemma: derivatives of X_2 integrals} with respect to any of $p,q$ or $y$ results in a sum of functions of these same types; we can therefore iterate these operations to compute higher order derivatives.  The following Lemma keeps track of the effect on regularity of differentiating the sum of the terms in Lemma \ref{lemma: derivatives of X_2 integrals}.


\begin{lemma}[More iterated derivative bounds]
\label{lemma: differentiating i=2 level set integrals}	
Set 
$$
B_r :=C^{r,1}(\bar X' \times \bar Y' \times \bar P' \times \bar Q')\times C^{r,1}((\bar X' \cap \partial X)  \times \bar Y' \times \bar P' \times \bar Q')
$$
and consider the operators $M^p, M^q, M^y: (B_r)^3 \rightarrow (B_{r-1})^3$ defined by
\begin{eqnarray*}
M^p:&&(a,b,c,a^\partial, b^\partial, c^\partial) \mapsto \\
&&\Big(a_p+\nabla \cdot (a\hat n_W)w+ bw -  \nabla_{W_=} \cdot (\frac{aw \hat n_W \cdot \hat n_Z}{\sqrt{1 - (\hat n_W \cdot \hat n_Z)^2}} \hat n_{W_=,Z}) + \nabla_{W_=} \cdot (cw \hat n_{W_=,Z}) , b_p,c_p,\\
&&-\frac{aw \hat n_W \cdot \hat n_X}{\sqrt{1 - (\hat n_W \cdot \hat n_X)^2}}+ a^\partial_p + \frac{w}{\sqrt{1-(\hat n_W \cdot \hat n_X)^2}} \nabla_{\partial X} \cdot (a^\partial \hat n_{\partial, W})+\frac{b^\partial w}{\sqrt{1-(\hat n_W \cdot \hat n_X)^2}}\\
&& - \frac{aw \hat n_W \cdot \hat n_Z}{\sqrt{1 - (\hat n_W \cdot \hat n_Z)^2}}\hat n_{W_=,Z} \cdot \hat n_{W_=,X} +cw\hat n_{W_=,Z} \cdot \hat n_{W_=,X}\\
&& - \nabla_{\bar W_= \cap \partial X}\cdot ( a^\partial \hat n_{\partial, W}\cdot \hat n_{\partial, Z}\frac{w}{\sqrt{[1-(\hat n_W \cdot \hat n_X)^2][1-(\hat n_W \cdot \hat n_Z)^2]}} \hat n_{\bar W_= \cap \partial X,Z})\\
&&+ \nabla_{\bar W_= \cap \partial X}\cdot ( c \frac{w}{\sqrt{[1-(\hat n_W \cdot \hat n_X)^2][1-(\hat n_Z \cdot \hat n_W)^2]}} \hat n_{\bar W_= \cap \partial X,Z}), b^\partial_p, c_p^\partial\Big),
\end{eqnarray*}

\begin{eqnarray*}
M^q:&&(a,b,c,a^\partial, b^\partial, c^\partial) \mapsto \\
&&\Big(a_q +\nabla_{W_=}\cdot(\frac{az}{\sqrt{1 - (\hat n_Z \cdot \hat n_W)^2}}\hat n_{W_=,Z}) -\nabla_{W_=}\cdot (\frac{cz\hat n_Z \cdot \hat n_W}{{\sqrt{1 - (\hat n_W \cdot \hat n_Z)^2}}}\hat n_{W_=,Z}),\\
&&b_q,bz+c_q+\nabla \cdot (c\hat n_Z)z,\\
&&a^\p_q-\frac{az}{\sqrt{1 - (\hat n_Z \cdot \hat n_W)^2}}\hat n_{W_=,Z}\cdot \hat n_{W_=,X} +\frac{cz\hat n_Z \cdot \hat n_W}{{\sqrt{1 - (\hat n_W \cdot \hat n_Z)^2}}}\hat n_{W_=,Z}\cdot \hat n_{W_=,X} \\
&&-\nabla_{\overline{W_=}\cap \partial X} \cdot( c^\p\hat n_{\partial, Z}\cdot \hat n_{\partial, W}\frac{z}{\sqrt{[1-(\hat n_Z \cdot \hat n_X)^2][1-(\hat n_Z \cdot \hat n_W)^2]}}\hat n_{\bar W_=\cap \partial X,Z})\\
&&\nabla_{\bar W_= \cap \p X}\cdot( a^\p \frac{z}{\sqrt{[1-(\hat n_Z \cdot \hat n_X)^2][1-(\hat n_W \cdot \hat n_Z)^2]}}d{\mathcal H}^{m-3}(x))\hat n_{\bar W_=\cap \partial X,Z},b^\p_q,\\
&&- \frac{cz \hat n_Z \cdot \hat n_X}{\sqrt{1 - (\hat n_Z \cdot \hat n_Z)^2}} +c^\partial_q+\frac{b^\p z}{\sqrt{1-(\hat n_Z \cdot \hat n_X)^2}} +\frac{z}{\sqrt{1-(\hat n_Z \cdot \hat n_X)^2}}\nabla_{\p X}\cdot (c^\p\hat n_{\p, Z})\Big)
\end{eqnarray*}

and

\begin{eqnarray*}
M^y:&&(a,b,c,a^\partial, b^\partial, c^\partial) \mapsto \\
&&\Big(a_y-  \nabla \cdot (a\hat n_W) w s_{yy} - b ws_{yy}- c \frac{\p \hat n_Z}{\p y} \cdot \hat n_W ,\\
&& \nabla \cdot (a \frac{\p \hat n_W}{\p y})+b_y+\nabla \cdot (c \frac{\p \hat n_Z}{\p y})  ,\\
&& -  a \frac{\p \hat n_W}{\p y} \cdot \hat n_Z - b zs_{yyy} +c_y-  \nabla \cdot (c\hat n_Z) z s_{yyy}  ,\\
&& \frac{aws_{yy} \hat n_W \cdot \hat n_X}{\sqrt{1 - (\hat n_W \cdot \hat n_X)^2}}+a^\p_y+\frac{ws_{yy}}{\sqrt{1-(\hat n_W \cdot \hat n_X)^2}} \nabla_{\partial X} \cdot (a^\p \hat n_{\partial, W})- \frac{b^\p ws_{yy}}{\sqrt{1-(\hat n_W \cdot \hat n_X)^2}}-\\
&&c^\p\frac{\partial \hat n_{\partial, Z}}{\partial y}\cdot \hat n_{\partial, W} ,\\
&&- a \frac{\p \hat n_W}{\p y} \cdot \hat n_X-c \frac{\p \hat n_Z}{\p y} \cdot \hat n_X+\nabla_{\partial X} \cdot (a^\p 
\frac{\partial \hat n_{\partial, W}}{\partial y})+b^\p_y+\nabla_{\partial X} \cdot (c^\p
\frac{\partial \hat n_{\partial, Z}}{\partial y}),\\
&& +  \frac{czs_{yyy} \hat n_Z \cdot \hat n_X}{\sqrt{1 - (\hat n_Z \cdot \hat n_X)^2}} - a^\p\frac{\partial \hat n_{\partial, W}}{\partial y}\cdot \hat n_{\partial, Z}- \frac{b^\p zs_{yyy}}{\sqrt{1-(\hat n_Z \cdot \hat n_X)^2}}+c^\p_y \\
&&-\frac{zs_{yyy}}{\sqrt{1-(\hat n_Z \cdot \hat n_X)^2}} \nabla_{\partial X} \cdot (c^\p \hat n_{\partial, Z}\Big)\\
\end{eqnarray*}
Then the norms $||M^p||$, $||M^q||$ and $||M^y||$ are controlled by non-degeneracy, transversality, linear independence, $||s_{y}||_{C^{r+2,1}}$ and $||\hat n_X||_{C^{r,1}}$.
\end{lemma}
\begin{proof} It is straightforward to compute:
\begin{eqnarray*}
	||M^p|| \leq 
	&&1+||\hat n_W||_{C^{r,1}}||w||_{C^{r,1}}+ ||w||_{C^{r-1}} + ||\frac{w \hat n_W \cdot \hat n_Z}{\sqrt{1 - (\hat n_W \cdot \hat n_Z)^2}} \hat n_{W_=,Z}||_{C^{r,1}} + ||w \hat n_{W_=,Z}||_{C^{r,1}} + 1+1\\
	&&+||\frac{w \hat n_W \cdot \hat n_X}{\sqrt{1 - (\hat n_W \cdot \hat n_X)^2}}||_{C^{r-1,1}}+ 1 +|| \frac{w}{\sqrt{1-(\hat n_W \cdot \hat n_X)^2}}||_{C^{r-1,1}}  ||\hat n_{\partial, W}||_{C^{r,1}}\\
	&&+||\frac{ w}{\sqrt{1-(\hat n_W \cdot \hat n_X)^2}}||_{C^{r-1,1}}
	 + ||\frac{w \hat n_W \cdot \hat n_Z}{\sqrt{1 - (\hat n_W \cdot \hat n_Z)^2}}\hat n_{W_=,Z} \cdot \hat n_{W_=,X} +||w\hat n_{W_=,Z} \cdot \hat n_{W_=,X}||_{C^{r-1,1}}\\
	&& + ||\hat n_{\partial, W}\cdot \hat n_{\partial, Z}\frac{w}{\sqrt{[1-(\hat n_W \cdot \hat n_X)^2][1-(\hat n_W \cdot \hat n_Z)^2]}} \hat n_{\bar W_= \cap \partial X,Z}||_{C^{r,1}}\\
	&&+ || \frac{w}{\sqrt{[1-(\hat n_W \cdot \hat n_X)^2][1-(\hat n_Z \cdot \hat n_W)^2]}} \hat n_{\bar W_= \cap \partial X,Z})||_{C^{r,1}} +1+1.
\end{eqnarray*}
Similar estimates hold for $M^q$ and $M^y$, and it is straightforward to see that the upper bounds are controlled by the indicated quantities.
\end{proof}


We are now ready to prove the Theorem \ref{thm: smoothness of G_2}, on the regularity of $G_2$.

\begin{proof}
	The proof is similar to the proof of Theorem \ref{thm: smoothness of G}; for indices $\alpha, \beta, \gamma$ with $\alpha + \beta +\gamma \leq r$, we apply the iterated operators $(M^y)^\alpha (M^p)^\beta(M^q)^\gamma$ to $(h, 0, 0,0,0,0)$, where $h=h(x,y,q) =f(x)\frac{q- s_{yy(x,y)}}{| D_{x} s_y(x,y)|} \in C^{r,1}$ is the integrand in $G_2$.  Setting
	$$
	(M^y)^\alpha (M^p)^\beta(M^q)^\gamma(h, 0, 0,0,0,0)=(a,b,c,a^\partial,b^\partial,c^\partial) \in B^{r-(\alpha +\beta+\gamma)},
	$$
	we have that
\begin{eqnarray*}
	\frac{\partial^{\alpha+\beta+\gamma} G_2}{\p y^\alpha \p p^\beta  \p q^\gamma}& =&\int_{W_= \cap Z_\le}ad\mathcal H^{m-1}(x)+\int_{W_\le \cap Z_\le}bd\mathcal H^{m}(x)+\int_{W_\le \cap Z_=}cd\mathcal H^{m-1}(x)\\
	&+&\int_{(\overline{W_= \cap Z_\le})\cap \p X}a^\p d\mathcal H^{m-2}(x)+\int_{(\overline{W_\le \cap Z_\le})\cap \p X}b^\p d\mathcal H^{m-1}(x)+\int_{(\overline{W_\le \cap Z_=})\cap \p X}c^\p d\mathcal H^{m-2}(x)
\end{eqnarray*}
is Lipschitz by Lemma  \ref{lemma: derivatives of X_2 integrals}, and it's norm is controlled by the sizes of the domains of integration and the $L^\infty$ norms of the integrands, which are in turn controlled by the desired quantities as a consequence of iterating Lemma \ref{lemma: differentiating i=2 level set integrals}.  As in Theorem \ref{thm: smoothness of G} in the previous section, and the corresponding result for one dimensional targets in \cite{ChiapporiMcCannPass17}, we observe that since the initial function $(h,0,0,0,0,0)$ we apply the operators to does not include any boundary terms, the  norm of the first application depends on $||\hat n_X||_{C^{r-1,1}}$  rather than $||\hat n_X||_{C^{r,1}}$, saving a derivative of smoothness in $\hat n_X$ in the final result.
\end{proof}



\begin{corollary}[Boostrapping smoothness for local ODE]
\label{cor: C11 to C21}
	Assume that the conditions in Theorem \ref{thm: smoothness of G_2} hold for  $Y',$ $P' =v'(Y')$ and $Q'=v''(Y')$ for some $r \geq 0$,  that $g \in C^{r,1}(Y')$    is bounded from below on $Y'$, $g \geq L_g >0$,  and that $f \in C^{r,1}(X')$ is bounded from above and below  $ \infty> U_f\geq f \geq L_f >0$ on $X' =\cup_{(y,p, q) \in Y' \times P'\times Q'} X_{2}(y,p,q)$.

	Then any almost everywhere solution $v \in C^{1,1}(Y')$ of the $i=2$ equation is in fact in $C^{r+2,1}(Y')$.
\end{corollary}	
\begin{proof}
	Setting $k(y) : =v'(y)$, we have $g(y) =G_2(y,k(y),k'(y))$ a.e.   At any $y$ where this holds, we must have $\mathcal{H}^{m-1}(X_2(y,k(y),k'(y)))\geq C >0$, where $C$ depends on $L_g$, $U_f$, $\min|D_xs_y|$, $\max |D_xs_{yy}|$ and the diameter of $X$. Noting that $$\frac{\partial G_2}{\partial q}(y,p,q) = \int_{X_2(y,p,q)}\frac{f(x)}{|D_xs_y(x,y)|}d\mathcal H^{m-1}(x),$$
	this yields a lower bound on $\frac{\partial G_2}{\partial q}(y,p,q)>B$ in the region of interest.

Therefore, by the Clarke inverse function theorem, $q \mapsto G_2(y,p,q)$ is invertible; denoting its inverse $q(y,p,\cdot)$, $q$ is as smooth as $G_2$ (that is, $q \in C^{r,1}$) and we have, almost everywhere
	\begin{equation}\label{eqn: k' ode}
	k'(y) =q(y,k(y),g(y)).
\end{equation}
The Lipschitz function $k$ is then equal to the antiderivative of its derivative; for a fixed $y_0$, we have
$$
k(y) - k(y_0) =\int_{y_0}^yk'(s)ds =\int_{y_0}^yq(s,k(s),g(s))ds.
$$ 
The fundamental theorem of calculus then implies that $k$ is \textit{everywhere} differentiable, and that \eqref{eqn: k' ode} holds \textit{for all} $y$.  In particular, $k'$ is Lipschitz as $y \mapsto q(y,k(y),g(y))$,  hence $v \in C^{2,1}(Y')$.
If $r >0$, one can immediately bootstrap to get $k' \in C^{r,1}(Y')$, hence $v \in C^{r+2,1}(Y')$.\end{proof}

\section{$C^{1,1}$ regularity of solutions to the $i=2$  ODE}	
\label{S:ODE} 

Throughout this section, we will assume that $n=1$ (one dimensional target), 
 and the minimizer $(u,v)= (v^s,u^{\tilde s})$ to \eqref{Kantorovich dual}
satisfies the local equation \eqref{local PDE}--\eqref{G_i} 
for  $i=2$ and a.e. $y \in Y$; this implies that $\partial ^sv(y) =X_2(y,v'(y),v''(y))$ for almost all $y$.

We will assume  $s, s_y \in C^2(\overline X \times \overline Y)$,
and the probability densities $f$ and $g$ are continuous on $X$ and $Y$,
satisfying bounds
\begin{eqnarray}\label{f bounds}
0 < L_f \le f(x)  \le U_f < \infty & \mbox{\rm for all} & x \in X 
\\ 0<L_g\leq g(y) \leq U_g <\infty &\mbox{\rm for all} & y \in Y. 
\label{g bounds}
\end{eqnarray}

As in Section \ref{sect: smoothness of $G_i$}, we assume without loss of generality that $s$ is everywhere convex with respect to $y$, so that $k(y) :=v'(y)$ is non-decreasing.

The functional
\begin{eqnarray}\label{eqn: 1-d G_i}G_2(y,p,q)
&:=&\int_{X_2(y,p,q)} \frac{q-s_{yy}(x,y)}{| D_{x} s_y(x,y)|}  
f(x) d{\mathcal H}^{m-1}(x).
\nonumber
\end{eqnarray}
is strictly increasing in $q$ on $\{(y,p,q) \mid G_2>0\}$ and diverges as $q \to \infty$. 

For $\beta >0$ and fixed $y$ and $ p$,  as in Section \ref{S:G2}, denote by $q(y, p, \beta )$ the unique solution of 

\begin{eqnarray}
q\mapsto G_2(y,p,q) =\beta
\nonumber
\end{eqnarray}
We therefore have $q(y,k(y),g(y)) =k'(y)$ almost everywhere.  Under the assumptions of Theorem \ref{thm: smoothness of G_2}  for some $r \geq 0$, $q(y,p, \beta )$ is Lipschitz continuous in all its arguments, by the Clarke implicit function theorem.  Although it holds only almost everywhere, this formulation provides some intuition for why we expect $k$ to be Lipschitz, since boundedness of $k$ and $g$ imply boundedness of $k'$ wherever the equation $q(y,k(y),g(y)) =k'(y)$ holds.


The estimate below essentially controls the volume of the region $F^{-1}([y_0,y_1])$ mapped to an interval $[y_0,y_1]$ by
the map $F$ of Theorem \ref{T:nonlocal PDE}
using the variation in $k$, compensated by a term reflecting the variation in $y$.  Together with mass balance, this
proposition easily implies that $k$ is Lipschitz 
 if continuous (and so $v \in C^{1,1}_{loc}$ if $v \in C^1_{loc}$); see Theorem \ref{T: Lipschitz k}.

The continuity of $k$, assumed in the proposition, will be confirmed for almost everywhere solutions to \eqref{G_i} below, 
 assuming an enhanced version of the twist hypothesis (see Assumption \ref{assumption: lower dimensional twist} and
Proposition \ref{P: continuous k}).

\begin{proposition}[The derivative of $v$ is Lipschitz if continuous]
\label{P: X_2 structure}
	Let $Y' \subset Y$ and $P' \subset P =\frac{\partial s}{\partial y}(X,Y)$ be regions such that $\inf_{y \in Y', p \in P'} \mathcal{H}^{m-1}(X_1(y,p)) >0$.  Then there exist positive constants $C_1$ and $C_2$ such that 
	
	\begin{equation}\label{eqn: bound on X_2 volume}
	{\mathcal H}^{m}\left(\mathop{\cup}_{y \in [y_0,y_1]}X_2(y,k(y),q(y,k(y), g(y)))\right) \geq C_1|k(y_0)-k(y_1)| -C_2|y_0-y_1|
	\end{equation}	
for any $y_0,y_1 \in Y'$ and monotone increasing, continuous function $k:Y' \rightarrow P'$.  
\end{proposition}

 Before proving the Proposition, it is instructive to provide some intuition.  For a \emph{fixed} $y$, the coarea formula yields
$\mathcal H^m(\cup_{p \in [p_0,p_1]}(X_1(y,p))) \sim |p_1-p_0|$,
where the constants of proportionality depend on two-sided bounds for $|D_x s_y|$ and $\mathcal H^{m-1}(X_1(y,p))$.  
The equation 
$$0< g(y) = \int_{X_2(y, p, q(y,p,g(y)))} \frac{q-s_{yy}(x,y)}{|D_x s_y|} f(x) d\mathcal H^{m}(x)
$$ 
forces $X_2(y,p,q(y,p,g(y)))$ to fill up a proportion 
of $X_1(y,p)$ which can be bounded in terms of the same bounds as before,  $L_g$, $U_f$, $\| s\|_{C^2}$, 
and $q(y,p,U_g)$. Thus there exists $C>0$ such that
$$\mathcal H^m \left(\bigcup_{p \in [p_0,p_1]}X_2(y,p, q(y,p,g(y)))\right) \geq  C|p_1-p_0|.$$

In the Proposition, $y$ is not fixed but varies within an interval $[y_0,y_1]$, and $p=k(y)$ is now a function.  Continuity and monotonicity force the image $k([y_0,y_1])$ to match the interval $[k_0 =k(y_0), k_1=k(y_1)]$.  If the domains $X_2(y,k(y),q(y,k(y), g(y))$ were independent of $y$, the result would then follow immediately, without the second term on the right hand side of \eqref{eqn: bound on X_2 volume}.

The second term compensates for the possibility that as $y$ and $k(y)$ change, the level curves bend in a way that reduces the volume on the left hand side.  

\begin{proof}
	Set $k_i =k(y_i)$ for $i=0,1$, and choose $C$ such that $|s_y(x,y_0) -s_y(x,y)|\leq C|y_0-y|$ for all $y \in Y'$ and $x  \in X_1(Y',P') $.

	Suppose that $k_0 \leq s_y(x,y_0) \leq k _1-C|y_1-y_0|$.  Then
	$$
	s_y(x,y_1) \leq s_y(x,y_0)+C|y_1-y_0|\leq k_1
	$$
	By the intermediate value theorem, $k(y)=s_y(x,y)$ for some $y \in [y_0,y_1]$; that is, $x \in X_1(y,k(y))$.
	
	Now, suppose in addition that  $q(y_0,k(y),g(y_0)) -s_{yy}(x,y_0) \geq \alpha >0$; by uniform continuity, there is a $\delta >0$ (depending on $\alpha$ but not $y$) such that, we have
	$$
	q(y,k(y),g(y)) -s_{yy}(x,y) \geq 0;
	$$ 
	that is, $x \in X_2(y,k(y),q(y,k(y),g(y)))$, if $|y-y_0| <\delta$.

	Now, note if $|y_0-y_1| \geq \delta$, the right hand side of \eqref{eqn: bound on X_2 volume} is negative for appropriate choices of $C_1,C_2$ (note that $k(y_0) -k(y_1)$ is less than or equal to the diameter of $P'$).  We can therefore assume $|y_0-y_1| <\delta$
	without loss of generality, and the above argument then yields $x \in X_2(y,k(y),q(y,k(y),g(y)))$ for some $y \in |y_0-y_1|$.
	
	It therefore follows that 
	\begin{eqnarray}\label{eqn: X_2 inclusion}
	&	\cup_{p \in[ k_0,k_1-C|y_1-y_0|]} \{x \in X_1(y_0,p) \mid q(y_0,p,g(y_0)) -s_{yy}(x,y_0) \geq \alpha \}\\
	&	\subseteq \cup_{y \in [y_0,y_1]}X_2(y,k(y),q(y,k(y), g(y)))\nonumber
	\end{eqnarray}
	
	
	Now   our definition of $q$ yields $g(y_0)  =G_2(y_0,p, q(y_0,p,g(y_0)))$, which
	implies that for a small enough $\alpha$, we have 
	$$
	\mathcal H^{m-1}(\{x \in X_1(y_0,k(y)): q(y_0,k(y),g(y_0)) -s_{yy}(x,y_0) \geq \alpha\} )\geq B\mathcal H^{m-1}(X_1(y_0, k(y)))
	$$ for some $B >0$ depending on the lower bound $L_g$ for $g$, $\min|D_xs_y|$, $\max|D_xs_{yy}|$ and the size of the level sets, $\sup_{(y,p) \in Y \times k(Y)} \mathcal H^{m-1}(\overline{X_1(y,p)})$.  It then follows that 
	\begin{eqnarray}\label{eqn: X_2 bound as proportion of X_1}
	&&vol[	\cup_{p \in[ k_0,k_1-C|y_1-y_0|]} \{x \in X_1(y_0,p) \mid q(y_0,p,g(y_0)) -s_{yy}(x,y_0) \geq \alpha \}] \nonumber\\
	&\geq &Bvol	[\cup_{p \in[ k_0,k_1-C|y_1-y_0|]}  X_1(y_0,p)]
	\end{eqnarray}
	
	Now, if $k_1-2C|y_1-y_0|\leq k_0$, then $|k_1-k_0|- 2C|y_1-y_0| <0$ and there is nothing to prove, since the right hand side of \eqref{eqn: bound on X_2 volume} is negative for appropriate choices of the constants. 
	
	On the other hand, if $k_1-2C|y_1-y_0|\geq  k_0$, then $k_1-C|y_1-y_0|-k_0 \geq \frac{|k_1-k_0|}{2}$, and so
	$$
	vol[	\cup_{p \in[ k_0,k_1-C|y_1-y_0|]}  X_1(y_0,p)] \geq D\frac{|k_1-k_0|}{2},
	$$
	where $D$ depends on the size $\min_{p \in [k_0,k_1]} \mathcal H^{m-1}(\overline{X_1(y_0,p)})$ of the level sets, and the speed limit $\min_{s_y(y_0,x) \in [k_0,k_1] }|D_xs_y(y_0,x)|$.
	This combined with \eqref{eqn: X_2 inclusion} and \eqref{eqn: X_2 bound as proportion of X_1}  establishes the result.
\end{proof}



Our strategy is to combine Proposition \ref{P: X_2 structure} with mass balance to deduce a Lipschitz condition on $k$. To apply this 
Proposition, we must first show that $k$ is continuous.  We do this under the following strengthening of the twist condition:


\begin{assumption}[Enhanced twist]\label{assumption: lower dimensional twist}
	 We say $s \in C^2(\overline X \times \overline Y)$ satisfies the enhanced twist condition if $x \in X$ and $y,\bar y \in Y$ imply
	$$
	D_xs(x,y) - D_xs(x, \bar y)
	$$ 
	cannot be a multiple of $D_{x} s_y(x,\bar y)$ unless $y=\bar y$.
\end{assumption}
Note that the usual twist condition asserts injectivity of the mapping $y \mapsto D_xs(x,y)$;   injectivity of the projection of $D_xs(x,y)$ onto the potential level sets $X_1(y,k)$ of the optimal map is sufficient to imply our enhanced twist condition. 
\begin{lemma}[Map continuity on interior of isodestination set]
\label{lemma: continuity at interior}
	Under the enhanced twist condition, the optimal map is uniquely defined (and therefore continuous) at any  point $x$ in the relative interior of $\partial ^sv(\bar y)$ in $X_1(\bar y,p)$.  
\end{lemma}
\begin{proof}
 If $x$ lies in the interior of $\p^s v(\bar y)$ relative to $X_1(\bar y, p)$ then
$u(\tilde x) =s(\tilde x,\bar y) -v(\bar y )$ for all $\tilde x \in X_1(\bar y, p)$ sufficiently close to   $x$.  
Therefore, $u$ is smooth along $X_1(\bar y, p)$ and differentiating we have
	$$
	D_{ X_1(\bar y, p)}u(x) =D_{ X_1(\bar y, p)} s( x,\bar y).
	$$
	On the other hand, if there is another $ y \neq \bar y$ such that $x \in \partial ^sv( y)$, the envelope condition yields
	
	$$
	D_{ X_1(\bar y, p)}u(x) =D_{ X_1(\bar y, p)} s( x, y)
	$$
	and so
	
	$$
	D_{ X_1(\bar y, p)} s( x, y)=D_{ X_1(\bar y, p)} s( x,\bar y),
	$$
	violating the enhanced twist condition.
\end{proof}

\begin{proposition}[Continuous differentiability of $v$]
\label{P: continuous k}
 Let $(v^s,v)$ achieve the minimum \eqref{Kantorovich dual} and solve \eqref{local PDE}--\eqref{G_i} a.e.\ on $Y$ with $i=2$. 
Let $Y' \subseteq Y$ and $X'':=\{x: d(x,X') <\delta \}$ be a neighbourhood of  $X' :=(s$-$\exp \circ D v^s)^{-1}(Y') $ for some $\delta >0$.  If $\|s_y\|_{C^2(\bar  X'' \times \bar Y')}<\infty$,
the enhanced twist condition and bounds \eqref{f bounds}--\eqref{g bounds} on the continuous probability densities $f$ and $g$  imply
continuity of $k=v'$ on $Y'$.
\end{proposition}

\begin{proof}
 Without loss of generality, assume $s_{yy} \ge 0$ so that $k$ is monotone increasing as before.
We  need only rule out jump discontinuities.	
Suppose $k$ has a jump discontinuity at $\bar y$, with left and right limits $k_0$ and $k_1$, respectively.
	
	For any $y$ where $k$ is differentiable with
	$g(y) =G_2(y, k(y), k'(y))$ and $X_2(y, k(y), k'(y)) =\partial^s  v(y)$,    the mean value theorem
	for integrals yields $x \in \p^s v(y)$ at which
$$
0< L_g \le g(y) = \frac{k'(y) - s_{yy}(x,y)}{|D_x s_y|} f(x) \mathcal H^{m-1}[X_2(y,k(y),k'(y))].
$$
Thus 	
$k'(y) -s_{yy}(x,y) \geq  \beta$ holds throughout a ball of radius $r$ in $X$, where  
$\beta:=\frac{c L_g }{ 2CU_f}>0$ with $c:=\min|D_xs_{y}(x,y)|$,
	$C:=\sup_{(y,k) \in Y \times k(Y)} \mathcal H^{m-1}(\overline{X_1(y,k)})$, and
	 $r$ depends only on $\delta $ and $B:=\sup_{(x,y) \in X" \times Y'}|D_xs_{y  y}(x,y)|$.
	
	Now take a sequence $\{y_i\}$ with $y_i <\bar y$ for which this is true, converging to $\bar y$. We have that $k(y_i) \rightarrow k_0$ and, after passing to a subsequence, the centers $x_i$ of the corresponding balls converge to an $\bar x \in X_1(\bar y, k_0) \cap \partial ^sv(y)$.  By continuity, we have $q(\bar y, k_0, g(\bar y)) -s_{yy}(\bar y,\bar x) \geq { 2\beta}>0$, and therefore,  $q(y, k(y), g(y)) -s_{yy}( y, x) \geq { \beta}>0$ for all
	$(x,y)$ close to $(\bar x, \bar y)$ with $y < \bar y$.

	Therefore,
	\begin{equation}\label{eqn: local uniform s-convexity}
	k'(y) -s_{yy}( y, x) \geq  \beta>0
	\end{equation}
	for all $x$ near $\bar x$, and almost all $y <\bar y$ near $\bar y$.

	 In addition, since $B_r(x_i) \cap X_1(y_i, k(y_i)) \subseteq \partial^sv(y_i)$, and each $x \in  B_r(\bar x) \cap X_1(\bar y, k_0)$ 
	 can be approximated by points $z_i(x) \in B_r(x_i) \cap X_1(\bar y_i, k(y_i))$, we can pass to the limit in the equality $u(z_i(x)) +v(y_i)=s(z_i(x), y_i)$ to obtain $u(x) +v(\bar y)=s(x, \bar y)$; that is, $x \in \partial^sv(\bar y)$.  Therefore, $B_{ r}(\bar x) \cap X_1(\bar y, k_0) \subset \partial^sv(\bar y)$. 

	 We have now shown that $\bar x$ in the relative interior of $\p^s v(\bar y)$ in $X_1(\bar y, k_0)$.
	Lemma \ref{lemma: continuity at interior}  therefore implies that the optimal map $F$ is continuous at $\bar x$.
	We  next show that all points in the open set $X_>(\bar y, k_0)  := \{ x \in X \mid s_y(x,\bar y) > k_0 \}$ 
	sufficiently near $\bar x$ must get mapped to $\bar y$; this violates mass balance and establishes the result.
	
	 Choose $x$ with $s_y(x, \bar y) >k_0$ such that $|\bar x -x| < \epsilon$,  and set $y = F(x)$. 
	The continuity of $F$ at $\bar x$ ensures $y$ is close to $\bar y$;  for $\epsilon>0$ sufficiently small
         we shall prove it must actually be equal to $\bar y$.  First observe
	$s_y(x, y) <k_1$  for $\epsilon>0$ sufficiently small, since $s_y(\bar x, \bar y) =k_0 <k_1$.
	
	If $y>\bar y$, then $k(y) >k_1$.  In this case, $s_y(x, y)=k(y) >k_1$, immediately yielding a contradiction.
	
	On the other hand, if $y<\bar y$, then \eqref{eqn: local uniform s-convexity} implies that
	\begin{eqnarray*}
		k_0-s_y(\bar x, \bar  y) - [k(y) -s_y(\bar x, y)] &\geq& \int_y^{\bar y} [k'(s) -s_{yy}(\bar x, s)]ds\\
		& \geq &  \beta|\bar y - y|	
	\end{eqnarray*}
	
	As $k_0 =s_y(\bar x, \bar  y)$, this means, for almost every $y <\bar y$, with $y$ close to $\bar y$
	$$
	s_y(\bar x, y)-s_y(x, y) =s_y(\bar x, y)-k(y)  \geq  \beta |\bar y - y|.
	$$
	Therefore,
	
	\begin{eqnarray*}
		s_y(\bar x, \bar y) -s_y(x, \bar y ) &=& 	s_y(\bar x,  y) -s_y(x, y ) + \int_y^{\bar y} [s_{yy}(\bar x,t) -s_{yy}(x,t)]dt \\
		&\geq & { \beta} |\bar y -y|-B|\bar y -y||\bar x -x|\\
		&=& |\bar y -y|({ \beta}-B|\bar x -x|)>0.
	\end{eqnarray*}
	for $|x-\bar x|$ sufficiently small.  This contradicts that assumption   $x \in X_{>}(\bar y,k_0)$.
	
	To summarize, we have shown that for $x \in X_{>}(\bar y,k_0)$ close to $\bar x$, we cannot have $F(x) >\bar y$ or $F(x) <\bar y$; we must therefore have $F(x) =\bar y$.  As this set has positive mass, and $\nu(\{\bar y\}) =0$, this violates mass balance, establishing that $k$ cannot have a jump discontinuity.
\end{proof}

\begin{theorem}[Lipschitz differentiability of $v$]
\label{T: Lipschitz k}
 Fix open sets $X \subset \R^m$ and $Y \subset \R$ equipped with continuous probability densities which are bounded 
away from zero and infinity.
Let $(v^s,v)$ achieve the minimum \eqref{Kantorovich dual} and solve \eqref{local PDE}--\eqref{G_i} a.e.\ on $Y$ with $i=2$.
Let $Y' \subset Y$ and  $P' = v'(Y')$ be regions such that $\inf_{y \in Y', p \in P'} \mathcal{H}^{m-1}(X_1(y,p)) >0$  and $X''$ as in Proposition \ref{P: continuous k}.  
Under the enhanced twist condition, if $\|s_y\|_{C^2(\bar  X'' \times \bar Y')}<\infty$ then $v \in C^{1,1}(Y')$.	
\end{theorem}

\begin{proof}
	Setting $k=v'$ yields $k'(y) =v''(y) = q(y,k(y), g(y))$ almost everywhere.  Choose $y_0 <y_1$ and denote $k(y_i)=k_i$ for $i=1,2$.  Since $\partial ^sv(y) =X_2(y,k(y), k'(y)) =X_2(y,k(y), q(y,k(y), g(y)))$ for a.e. $y$, mass balance and 
	Propositions \ref{P: X_2 structure}   and \ref{P: continuous k} combine to imply
	\begin{eqnarray}U_g|y_1-y_0|&{ \ge}&
	\int_{y_0}^{y_1}g(y)dy\\
	& { =} & \int_{\cup_{y \in [y_0,y_1]}X_2(y,k,q(y,k(y), g(y)){ )}}f(x)d{\mathcal H}^{m}(x)\\
	&\geq& L_fvol[\cup_{y \in [y_0,y_1]}X_2(y,k,q(y,k(y), g(y)){ )}]\\
	&\geq &L_f[ C_1|k(y_0)-k(y_1)| -C_2|y_0-y_1|].
	\end{eqnarray}
	
	This is the desired conclusion.
\end{proof}

\begin{remark}\label{rem: higher order smoothness}
	Combined with Corollary \ref{cor: C11 to C21}, this yields conditions under which any solution to the $i=2$ local equation is $C^{2,1}$, or in fact smoother, depending on $G_2$
\end{remark}

\begin{appendices}
	\section{Formulas for partial derivatives}
	\label{appendix: derivatives}
	
	The partial derivatives of the functions defined in Lemma \ref{lemma: derivatives of X_2 integrals} are given almost everywhere by the following formulas, with $w(x,y)=|D_x s_y|^{-1}$:
	\begin{eqnarray*}
		A_p - &&\int_{X_2} a_p d{\mathcal H}^{m-1} = \frac{\p}{\p \tilde p}\bigg|_{\tilde p= p} \int_{W_=(y,\tilde p) \cap Z_\le(y,q)} 
		a(x,y,p,q) \hat n_W \cdot \hat n_W d{\mathcal H}^{m-1}(x)
		\\ &=&  \frac{\p}{\p \tilde p}\bigg[ \int_{W_\le \cap Z_\le} \nabla \cdot (a\hat n_W)  d{\mathcal H}^{m}-
		\int_{W_\le \cap Z_=} a \hat n_W \cdot \hat n_Z d{\mathcal H}^{m-1}
		- \int_{\overline{{W_\le \cap Z_\le}} \cap \p X} a \hat n_W \cdot \hat n_X d{\mathcal H}^{m-1}
		\bigg]_{\tilde p=p}
		\\ &=& \int_{W_= \cap Z_\le} \nabla \cdot (a\hat n_W) w d{\mathcal H}^{m-1}
		- \int_{W_=\cap Z_=}  \frac{aw \hat n_W \cdot \hat n_Z}{\sqrt{1 - (\hat n_W \cdot \hat n_Z)^2}} d{\mathcal H}^{m-2}
		\\ && - \int_{\overline{{W_= \cap Z_\le}} \cap \p X} \frac{aw \hat n_W \cdot \hat n_X}{\sqrt{1 - (\hat n_W \cdot \hat n_X)^2}}d{\mathcal H}^{m-2},
	\end{eqnarray*}
	
	\begin{eqnarray*}
		A_q - &&\int_{X_2} a_q d{\mathcal H}^{m-1}= \frac{\p}{\p \tilde q}\bigg|_{\tilde q= q} \int_{W_=(y,p) \cap Z_\le(y,\tilde q)} a(x,y,p,q)  d{\mathcal H}^{m-1}(x)
		\\ &=&  \int_{W_= \cap Z_=} \frac{az}{\sqrt{1 - (\hat n_Z \cdot \hat n_W)^2}}d{\mathcal H}^{m-2},
	\end{eqnarray*}

	\begin{eqnarray*}
		A_y - &&\int_{X_2} a_y d{\mathcal H}^{m-1} + \int_{W_= \cap Z_=} \frac{a z s_{yyy }}{\sqrt{1-(\hat n_Z \cdot \hat n_W)^2}}  d{\mathcal H}^{m-2}
		= \frac{\p}{\p \tilde y}\bigg|_{\tilde y= y} \int_{W_=(\tilde y,p) \cap Z_\le(y,q)} 
		a(x,y,p,q) d{\mathcal H}^{m-1}(x) 
		\\ &=&  \frac{\p}{\p \tilde y}\bigg[ \int_{\tilde W_\le \cap Z_\le} \nabla \cdot (a\hat n_{\tilde W})  d{\mathcal H}^{m}-
		\int_{{\tilde W}_\le \cap Z_=} a \hat n_{\tilde W} \cdot \hat n_Z d{\mathcal H}^{m-1}
		- \int_{\overline{{{\tilde W}_\le \cap Z_\le}} \cap \p X} a \hat n_{\tilde W} \cdot \hat n_X d{\mathcal H}^{m-1}
		\bigg]_{\tilde y=y}
		\\ &=& \int_{W_\le \cap Z_\le} \nabla \cdot (a \frac{\p \hat n_W}{\p y})  d{\mathcal H}^{m}
		- \int_{{W}_\le \cap Z_=} a \frac{\p \hat n_W}{\p y} \cdot \hat n_Z d{\mathcal H}^{m-1}
		- \int_{\overline{{{W}_\le \cap Z_\le}} \cap \p X} a \frac{\p \hat n_W}{\p y} \cdot \hat n_X d{\mathcal H}^{m-1}
		\\ && - \int_{W_= \cap Z_\le} \nabla \cdot (a\hat n_W) w s_{yy}  d{\mathcal H}^{m-1}
		+ \int_{W_=\cap Z_=}  \frac{aws_{yy} \hat n_W \cdot \hat n_Z}{\sqrt{1 - (\hat n_W \cdot \hat n_Z)^2}} d{\mathcal H}^{m-2}
		\\ && + \int_{\overline{{W_= \cap Z_\le}} \cap \p X} \frac{aws_{yy} \hat n_W \cdot \hat n_X}{\sqrt{1 - (\hat n_W \cdot \hat n_X)^2}}d{\mathcal H}^{m-2}
	\end{eqnarray*}
	where $\frac{\p}{\p y} \hat n_W =  (\hat n_Z - \hat n_W (\hat n_W \cdot \hat n_Z))w/z$,

	\begin{eqnarray*}
		B_p=\int_{W_\leq(y,p)\cap Z_\leq(y,q)} b_p d{\mathcal H}^{m}(x) +\int_{W_=(y,p)\cap Z_\leq(y,q)} b wd{\mathcal H}^{m-1}(x),
	\end{eqnarray*}
	
	\begin{eqnarray*}
		B_q=\int_{W_\leq(y,p)\cap Z_\leq(y,q)} b_q d{\mathcal H}^{m}(x) +\int_{W_\leq(y,p)\cap Z_=(y,q)} b zd{\mathcal H}^{m-1}(x),
	\end{eqnarray*}
	\begin{eqnarray*}
		B_y=\int_{W_\leq(y,p)\cap Z_\leq(y,q)} b_y d{\mathcal H}^{m}(x) -\int_{W_=(y,p)\cap Z_\leq(y,q)} b ws_{yy}d{\mathcal H}^{m-1}(x)-\int_{W_\leq(y,p)\cap Z_=(y,q)} b zs_{yyy}d{\mathcal H}^{m}(x),
	\end{eqnarray*}
	
	$$C_p =\int_{W_{\leq}(y,p) \cap Z_{=}(y,q)} c_pd{\mathcal H}^{m-1}(x) +\int_{W_{=}(y,p) \cap Z_{=}(y,q)} cwd{\mathcal H}^{m-2}(x),$$
	
	\begin{eqnarray*}
		C_q&=&\int_{W_{\leq}(y,p) \cap Z_{=}(y,q)} c_q d{\mathcal H}^{m-1}(x)+\int_{W_{\leq}(y,p) \cap Z_{=}(y,q)} \nabla \cdot (c\hat n_Z)z d{\mathcal H}^{m-1}(x)\\
		&&-\int_{W_{=}(y,p) \cap Z_{=}(y,q)} \frac{cz\hat n_Z \cdot \hat n_W}{{\sqrt{1 - (\hat n_{ \p, W} \cdot \hat n_{ \p, Z})^2}}} d{\mathcal H}^{m-2}(x)-  \int_{\overline{{W_\leq \cap Z_=}} \cap \p X} \frac{cz \hat n_Z \cdot \hat n_X}{\sqrt{1 - (\hat n_Z \cdot \hat n_X)^2}}d{\mathcal H}^{m-2},
	\end{eqnarray*}
	
	
	\begin{eqnarray*}
		C_y = &&\int_{W_\leq \cap Z_=} c_y d{\mathcal H}^{m-1} - \int_{W_= \cap Z_=} \frac{c w s_{yy }}{\sqrt{1-(\hat n_Z \cdot \hat n_W)^2}}  d{\mathcal H}^{m-2}
		\\ &+& \int_{W_\le \cap Z_\le} \nabla \cdot (c \frac{\p \hat n_Z}{\p y})  d{\mathcal H}^{m}
		- \int_{{W}_= \cap Z_\le} c \frac{\p \hat n_Z}{\p y} \cdot \hat n_W d{\mathcal H}^{m-1}
		- \int_{\overline{{{W}_\le \cap Z_\le}} \cap \p X} c \frac{\p \hat n_Z}{\p y} \cdot \hat n_X d{\mathcal H}^{m-1}
		\\ && - \int_{W_\le \cap Z_=} \nabla \cdot (c\hat n_Z) z s_{yyy}  d{\mathcal H}^{m-1}
		+ \int_{W_=\cap Z_=}  \frac{czs_{yyy} \hat n_W \cdot \hat n_Z}{\sqrt{1 - (\hat n_W \cdot \hat n_Z)^2}} d{\mathcal H}^{m-2}
		\\ && + \int_{\overline{{W_\le \cap Z_=}} \cap \p X} \frac{czs_{yyy} \hat n_Z \cdot \hat n_X}{\sqrt{1 - (\hat n_Z \cdot \hat n_X)^2}}d{\mathcal H}^{m-2},
	\end{eqnarray*}

			
	\begin{eqnarray*}
		A^\partial_p& =& \int_{(\overline{W_= \cap Z_{\leq})}\cap \partial X)} a^\p_p d{\mathcal H}^{m-2}(x) +\int_{(\overline{W_= \cap Z_{\leq})}\cap \partial X)} \frac{w}{\sqrt{1-(\hat n_W \cdot \hat n_X)^2}} \nabla_{\partial X} \cdot (a^\p \hat n_{\partial, W})d{\mathcal H}^{m-2}(x)\\
		&-&\int_{(\overline{W_= \cap Z_{=})}\cap \partial X)} a^\p\hat n_{\partial, W}\cdot \hat n_{\partial, Z}\frac{w}{\sqrt{[1-(\hat n_W \cdot \hat n_X)^2][1-(\hat n_{ \p, W} \cdot \hat n_{ \p, Z})^2]}} d{\mathcal H}^{m-3}(x),
	\end{eqnarray*}
	
	\begin{eqnarray*}
		A^\partial_q& =& \int_{(\overline{W_= \cap Z_{\leq})}\cap \partial X)} a^\p_q d{\mathcal H}^{m-2}(x) +\int_{(\overline{W_= \cap Z_{=})}\cap \partial X)} a^\p \frac{z}{\sqrt{[1-(\hat n_Z \cdot \hat n_X)^2][1-(\hat n_{ \p, W} \cdot \hat n_{ \p,Z})^2]}}d{\mathcal H}^{m-3}(x),\\
	\end{eqnarray*}
	
	\begin{eqnarray*}
		A^\partial_y& =&  \int_{(\overline{W_= \cap Z_{\leq})}\cap \partial X)} a^\p_y d{\mathcal H}^{m-2}(x) -\int_{(\overline{W_= \cap Z_{=})}\cap \partial X)} a^\p \frac{zs_{yyy}}{\sqrt{[1-(\hat n_Z \cdot \hat n_X)^2][1-(\hat n_{ \p,W} \cdot \hat n_{ \p,Z})^2]}}d{\mathcal H}^{m-3}(x)\\
		&+&\int_{(\overline{W_\leq \cap Z_{\leq})}\cap \partial X)} \nabla_{\partial X} \cdot (a^\p 
		\frac{\partial \hat n_{\partial, W}}{\partial y})
		d{\mathcal H}^{m-1}(x) -\int_{(\overline{W_\le \cap Z_{=}}\cap \partial X)} a^\p\frac{\partial \hat n_{\partial, W}}{\partial y}\cdot \hat n_{\partial, Z} d{\mathcal H}^{m-2}(x)\\
		& +&
		\int_{(\overline{W_= \cap Z_{\leq})}\cap \partial X)}\frac{ws_{yy}}{\sqrt{1-(\hat n_W \cdot \hat n_X)^2}} \nabla_{\partial X} \cdot (a^\p \hat n_{\partial, W})d{\mathcal H}^{m- 2}(x)
		\\ &-&\int_{(\overline{W_= \cap Z_{=})}\cap \partial X)}\frac{ws_{yy}}{\sqrt{[1-(\hat n_W \cdot \hat n_X)^2][1-(\hat n_{ \p,W} \cdot \hat n_{ \p,Z})^2]}} a^\p\hat n_{\partial, W}\cdot \hat n_{\partial, Z} d{\mathcal H}^{m-3}(x),
	\end{eqnarray*}	
where $\hat n_{\partial, W} :=\frac{\hat n_W -(\hat n_W \cdot \hat n_X)\hat n_X}{\sqrt{1-(\hat n_W \cdot \hat n_X)^2}}$ and $\hat n_{\partial, Z} :=\frac{\hat n_Z -(\hat n_Z \cdot \hat n_X)\hat n_X}{\sqrt{1-(\hat n_Z \cdot \hat n_X)^2}}$ are the outward unit normals to $\bar W_\le \cap \partial X$ and  $\bar Z_\le \cap \partial X$ in $\partial X$, respectively and
		
	
	$$
	B^\partial_p =\int_{(\overline{W_\le \cap Z_{\le})}\cap \partial X)} b^\p_p d{\mathcal H}^{m-1}(x) +\int_{(\overline{W_= \cap Z_{\le})}\cap \partial X)} \frac{b^\p w}{\sqrt{1-(\hat n_W \cdot \hat n_X)^2}} d{\mathcal H}^{m-2}(x),
	$$
	
	$$
	B^\partial_q =\int_{(\overline{W_\le \cap Z_{\le})}\cap \partial X)} b^\p_q d{\mathcal H}^{m-1}(x) +\int_{(\overline{W_\le\cap Z_{=})}\cap \partial X)} \frac{b^\p z}{\sqrt{1-(\hat n_Z \cdot \hat n_X)^2}} d{\mathcal H}^{m-2}(x),
	$$
	
	
	\begin{eqnarray*}
		B^\partial_y &=&\int_{(\overline{W_\le \cap Z_{\le})}\cap \partial X)} b^\p_y d{\mathcal H}^{m-1}(x) -\int_{(\overline{W_= \cap Z_{\le})}\cap \partial X)} \frac{b^\p ws_{yy}}{\sqrt{1-(\hat n_W \cdot \hat n_X)^2}} d{\mathcal H}^{m-2}(x)\\
		&-&\int_{(\overline{W_\le\cap Z_{=})}\cap \partial X)} \frac{b^\p zs_{yyy}}{\sqrt{1-(\hat n_Z \cdot \hat n_X)^2}} d{\mathcal H}^{m  -2}(x),
	\end{eqnarray*}
	
	\begin{eqnarray*}
		C^\partial_p& =& \int_{(\overline{W_\le \cap Z_{=})}\cap \partial X)} c^\p_p d{\mathcal H}^{m-2}(x) +\int_{(\overline{W_= \cap Z_{=})}\cap \partial X)} c^\p \frac{w}{\sqrt{[1-(\hat n_W \cdot \hat n_X)^2][1-(\hat n_{ \p,Z} \cdot \hat n_{ \p,W})^2]}}d{\mathcal H}^{m-3}(x),
	\end{eqnarray*}
	
	\begin{eqnarray*}
		C^\partial_q &=& \int_{(\overline{W_\le \cap Z_{=})}\cap \partial X)} c^\p_q d{\mathcal H}^{m-2}(x) +\int_{(\overline{W_\leq \cap Z_{=})}\cap \partial X)} \frac{z}{\sqrt{1-(\hat n_Z \cdot \hat n_X)^2}} \nabla_{\partial X} \cdot (c^\p \hat n_{\partial, Z})d{\mathcal H}^{m-2}(x)\\
		&-&\int_{(\overline{W_= \cap Z_{=})}\cap \partial X)} c^\p\hat n_{\partial, Z}\cdot \hat n_{\partial, W}\frac{z}{\sqrt{[1-(\hat n_Z \cdot \hat n_X)^2][1-(\hat n_{ \p,Z} \cdot \hat n_{ \p,W})^2]}} d{\mathcal H}^{m-3}(x),
	\end{eqnarray*}
	and
	
	\begin{eqnarray*}
		C^\partial_y& =&  \int_{(\overline{W_\le \cap Z_{=})}\cap \partial X)} c^\p_y d{\mathcal H}^{m-2}(x) -\int_{(\overline{W_= \cap Z_{=})}\cap \partial X)} c^\p \frac{ws_{yy}}{\sqrt{[1-(\hat n_W \cdot \hat n_X)^2][1-(\hat n_{ \p, W} \cdot \hat n_{ \p,Z})^2]}}d{\mathcal H}^{m-3}(x)\\
		&+&\int_{(\overline{W_\leq \cap Z_{\leq})}\cap \partial X)} \nabla_{\partial X} \cdot (c^\p
		\frac{\partial \hat n_{\partial, Z}}{\partial y})
		d{\mathcal H}^{m-1}(x) -\int_{(\overline{W_= \cap Z_{\le}}\cap \partial X)} c^\p\frac{\partial \hat n_{\partial, Z}}{\partial y}\cdot \hat n_{\partial, W} d{\mathcal H}^{m-2}(x)\\
		& -&\int_{(\overline{W_\leq \cap Z_{=})}\cap \partial X)}\frac{zs_{yyy}}{\sqrt{1-(\hat n_Z \cdot \hat n_X)^2}} \nabla_{\partial X} \cdot (c^\p \hat n_{\partial, Z})d{\mathcal H}^{m-1}(x)
		\\ &+&\int_{(\overline{W_= \cap Z_{=})}\cap \partial X)}\frac{zs_{yyy}}{\sqrt{[1-(\hat n_Z \cdot \hat n_X)^2][1-(\hat n_{ \p,W} \cdot \hat n_{ \p,Z})^2]}} c^\p\hat n_{\partial, W}\cdot \hat n_{\partial, Z} d{\mathcal H}^{m-3}(x).
	\end{eqnarray*}
	
		
		\end{appendices}

\bibliographystyle{plain}
\bibliography{newbib}

\end{document}